\documentclass[review,onefignum,onetabnum]{siamart171218}

\usepackage{lipsum}
\usepackage{amsfonts}
\usepackage{graphicx}
\usepackage{epstopdf}
\usepackage{algorithmic}
\usepackage{amssymb}
\usepackage[english]{babel}
\usepackage{array}
\usepackage{adjustbox}
\usepackage{relsize}
\usepackage{spverbatim}
\usepackage{listings}
\usepackage{mathrsfs}

\DeclareMathAlphabet{\mathpzc}{OT1}{pzc}{m}{it}

\newcolumntype{L}{>{$}l<{$}}

\ifpdf
  \DeclareGraphicsExtensions{.eps,.pdf,.png,.jpg}
\else
  \DeclareGraphicsExtensions{.eps}
\fi

\newsiamremark{hypothesis}{Hypothesis}
\crefname{hypothesis}{Hypothesis}{Hypotheses}
\newsiamthm{claim}{Claim}

\nolinenumbers

\headers{The zeros of Bessel functions}{T. M. Dunster}

\title{Uniform asymptotic expansions for the zeros of Bessel functions}

\author{T. M. Dunster\thanks{Department of Mathematics and Statistics, San Diego State University, 5500 Campanile Drive, San Diego, CA 92182-7720, USA. 
  (\email{mdunster@sdsu.edu}, \url{https://tmdunster.sdsu.edu}).}
  }

\usepackage{amsopn}

\makeatletter
\newcommand*{\addFileDependency}[1]{% argument=file name and extension
  \typeout{(#1)}% latexmk will find this if $recorder=0 (however, in that case, it will ignore #1 if it is a .aux or .pdf file etc and it exists! if it doesn't exist, it will appear in the list of dependents regardless)
  \@addtofilelist{#1}% if you want it to appear in \listfiles, not really necessary and latexmk doesn't use this
  \IfFileExists{#1}{}{\typeout{No file #1.}}% latexmk will find this message if #1 doesn't exist (yet)
}
\makeatother

\newcommand*{\myexternaldocument}[1]{%
    \externaldocument{#1}%
    \addFileDependency{#1.tex}%
    \addFileDependency{#1.aux}%
}
\nolinenumbers

\ifpdf
\hypersetup{
  pdftitle={Uniform asymptotic expansions for the zeros of Bessel functions},
  pdfauthor={T. M. Dunster}
}
\fi

\myexternaldocument{ex_supplement}

\begin{document}

\maketitle

\begin{abstract}
Reformulated uniform asymptotic expansions are derived for ordinary differential equations having a large parameter and a simple turning point. These involve Airy functions, but not their derivatives, unlike traditional asymptotic expansions. From these, asymptotic expansions are derived for the zeros of Bessel functions that are valid for large positive values of the order, uniformly for all the zeros. The coefficients in the expansions are explicitly given elementary functions, and similar expansions are derived for the zeros of the derivatives of Bessel functions.
\end{abstract}



\section{Introduction}
\label{sec1}

In this paper we obtain uniform asymptotic expansions for zeros of Bessel functions, via the differential equation they satisfy \cite[Eq. 10.2.1]{NIST:DLMF}. These functions have numerous mathematical and physical applications \cite[Sects. 10.72 and 10.73]{NIST:DLMF}, and their zeros also have many applications, such as in wave propagation, scattering theory and quantum mechanics \cite{Dunster:2006:RLC}, \cite{Elizalde:1993:SRZ}, \cite{Ferreira:2008:ZMF}, \cite{Liu:2007:ZBS}, \cite{Parnes:1972:CZM}.

Let us recap existing asymptotic expansions for the $m$th positive zero $j_{\nu,m}$ ($m=1,2,3,\ldots$) of the Bessel function of the first kind $J_{\nu}(z)$. Similar results to those below hold for other solutions of  Bessel's equation such as the Bessel function of the second kind $Y_{\nu}(z)$ and the Hankel functions $H_{\nu}^{(1)}(z)$ and $H_{\nu}^{(2)}(z)$, as well as their derivatives.

Firstly, the simplest are the McMahon’s asymptotic expansions for the later zeros \cite[Eq. 10.21.19]{NIST:DLMF}, where for bounded $\nu \geq 0$
%%%%%%%%%%%%%%%%%%%%%%%%%
\begin{equation}
%\label{}
\label{eq01}
j_{\nu,m} \sim
\left(m +\frac{1}{2}\nu -\frac{1}{4}\right)\pi
+\sum_{s=0}^{\infty}
\frac{p_{s}(\nu)}{\left\{\left(m +\frac{1}{2}\nu 
-\frac{1}{4}\right)\pi\right\}^{2s+1}}
\quad   (m \to \infty).
\end{equation}
%%%%%%%%%%%%%%%%%%%%%%%%%
Nemes \cite{Nemes:2021:PTC} proved that $p_{s}(\nu)$ is a polynomial in $\nu$ of order $2s$, and moreover for $\nu \in (-1/2,1/2)$ the $n$th error term does not exceed the first neglected term in absolute value and has the same sign as that term.

Next from \cite[Eq. 10.21.32]{NIST:DLMF} for bounded $m$
%%%%%%%%%%%%%%%%%%%%%%%%%
\begin{equation}
\label{eq02}
j_{\nu,m} \sim
\nu \sum_{s=0}^{\infty}
\frac{\alpha_{s}}{\nu^{2s/3}}
\quad (\nu \to \infty),
\end{equation}
%%%%%%%%%%%%%%%%%%%%%%%%%
where $\alpha_{s}$ are polynomials involving $\mathrm{a}_{m}$, the $m$th negative zero of $\mathrm{Ai}(x)$ ordered by increasing absolute values. Qu and Wong \cite{Qu:1999:BPU} showed that each zero is bounded below by the first two terms, and bounded above by the first three terms.

The paper of Olver \cite{Olver:1954:AEB} was a major contribution to the study of the \emph{uniform} asymptotic behavior of Bessel functions of large order, including in the later sections the real and complex zeros of Bessel functions. Using the now standard expansions for ODEs having a simple turning point in the complex plane, he derived asymptotic expansions for Bessel functions that involve Airy functions and their derivatives, and in particular the expansions (given with error bounds) in \cite[Chap. 11, Eq. (10.18)]{Olver:1997:ASF}. With the one for $J_{\nu}(z)$, which is valid for unbounded complex $z$ including the positive axis where the real zeros are located, in \cite[\S 7]{Olver:1954:AEB} he proved the existence of a uniform asymptotic expansion of the form
%%%%%%%%%%%%%%%%%%%%%%%%%
\begin{equation}
\label{eq03}
j_{\nu,m} \sim
\nu \sum_{s=0}^{\infty}
\frac{z_{m,s}}{\nu^{2s}}
\quad (\nu \to \infty, \, m=1,2,3,\ldots).
\end{equation}
%%%%%%%%%%%%%%%%%%%%%%%%%
Due to the dual uniformity of the parent Airy expansions, that is being valid for either one or both of the argument and order large, the expansion (\ref{eq03}) is also asymptotically valid for $m \to \infty$ with $\nu > 0$ fixed. Thus we we can regard (\ref{eq01}) and (\ref{eq02}) as special cases (with the exception $\nu =0$).

For each $m$ the coefficients $z_{m,s}$ depend on the parameter
%%%%%%%%%%%%%%%%%%%%%%%%%
\begin{equation}
\label{eq04}
\zeta_{m,0} = \nu^{-2/3} \mathrm{a}_{m}
\in (-\infty,0).
\end{equation}
%%%%%%%%%%%%%%%%%%%%%%%%%
The first term $z_{m,0}$ is readily computable via an implicit equation, but the subsequent coefficients are significantly more complicated than those in (\ref{eq01}) and (\ref{eq02}), and not given explicitly in his paper or elsewhere. For this reason it does not appear that this uniform asymptotic expansion has been used in the literature for approximating the zeros, despite it being the most powerful one in principle. Olver showed that similar expansions to (\ref{eq03}) hold for the zeros of the other Bessel functions (real and complex), as well as for their derivatives. 

His method used Taylor series expansions, which involved repeated differentiation of the coefficients in the Airy expansions described above. The difficulty in using these coefficients is that they involve nested integrations, and are hard to evaluate beyond the first few. Also the presence of the derivative of the Airy function further complicates matters.

The expansion (\ref{eq03}), and related ones for other solutions of Bessel's equation, is the main focus of this paper, with the aim to make (\ref{eq03}) easily computable, thus providing a means to compute all the zeros to high accuracy. Here we shall derive simple and explicit expressions for $z_{m,s}$, as described below, using a method that is widely applicable to other special functions.

We achieve this in two steps. Firstly in \cref{sec2} we examine general ordinary differential equations having a turning point. We shall obtain Cherry type expansions (see \cite{Cherry:1950:UAF}), where there is no derivative of the Airy approximant, and instead an asymptotic expansion occurs inside the argument of a single Airy function. As we shall demonstrate, this makes it significantly easier to derive asymptotic expansions for the zeros. Error bounds in the real variable case for expansions of Cherry type were later given by the present author in \cite{Dunster:2014:OEB}. However, the coefficients in both these papers are as hard to compute as in the standard form \cite[Chap. 11]{Olver:1997:ASF}. 

Instead, in the present paper we employ the exponential form of Liouville Green (LG) expansions as given in \cite{Dunster:2020:LGE}, and as a result our new Airy function expansions have coefficients that are readily computable, since they do not require nested integration. The method of employing these LG expansions bears similarities to \cite{Dunster:2017:COA}, where uniform asymptotic expansions of the form involving the Airy function and its derivative were obtained where the coefficients are also easier to compute than the traditional forms described above. In addition the error bounds for this approach \cite{Dunster:2021:SEB} are significantly simpler in form than given previously \cite[Chap. 11]{Olver:1997:ASF}.

We then apply the new general theory of \cref{sec2} to Bessel's equation in \cref{sec3}, which allows us to asymptotically solve for the zeros with an expansion of the form (\ref{eq03}), but now obtaining an explicit recursion relation for the coefficients. We further show that they are elementary rational functions of three readily computed variables. We also have the advantage of symbolic algebra to evaluate these coefficients, in our case Maple\footnote{Maple 2020.2. Maplesoft, a division of Waterloo Maple Inc., Waterloo, Ontario.}. In \cref{sec4} we obtain similar expansions for the zeros of the derivatives of Bessel functions.

The only restriction in our method is that the number of terms in the coefficients $z_{m,s}$ grows exponentially when given explicitly in terms of just the three variables, as we show in a Maple worksheet in which the first four terms are recorded (\cref{secA}). More terms than we give there can readily be evaluated by our method, and to avoid overflow they can be stored in a more general form (see (\ref{eq94}) - (\ref{eq97}) below), with the explicit dependence on the three variables stored in intermediary terms appearing in these expressions. This keeps the size more manageable, and allows computation of the zeros to a very high degree of accuracy, uniformly for $m=1,2,3,\ldots$. We illustrate the uniformity of our expansions with some numerical examples in \cref{sec5}.

Our method is quite general, and can be used to study the zeros of a wide range of special functions that are solutions of an ODE having a turning point. Moreover, in a subsequent paper we shall derive explicit error bounds when the series (\ref{eq03}) is truncated after three terms, and this should also be possible with the zeros of other functions. On the other hand, deriving error bounds for (\ref{eq03}) does not seem feasible using Olver's method.

\section{Airy function expansions in Cherry form}
\label{sec2}

Here we consider linear second order differential equations in the standard normalized form
%%%%%%%%%%%%%%%%%%%%%%%%%
\begin{equation}
\label{eq05}
d^{2}w/dz^{2}=\left\{u^{2}f(z) +g(z)\right\} w,
\end{equation}
%%%%%%%%%%%%%%%%%%%%%%%%%
where $u$ is a large positive parameter, and $z$ lies in a complex domain which may be unbounded. Extending our results to complex $u$ would not be a major task, but we assume $u$ real throughout this paper. 

Many special functions satisfy equations of the form (\ref{eq05}). The functions $f(z)$ and $g(z)$ are assumed to be meromorphic in a certain domain $Z$ (defined by shortly), and for ease of notation that are assumed to be independent of $u$ (this condition can easily be relaxed). We further assume $f(z)$ has a simple zero at $z=z_{0}$, which is the turning point of the equation.

Following \cite[Chap. 11]{Olver:1997:ASF} let $\xi$ and $\zeta$ be defined by
%%%%%%%%%%%%%%%%%%%%%%%%%
\begin{equation}
\label{eq06}
\xi=\frac{2}{3}\zeta ^{3/2}
= \int_{z_{0}}^{z} f^{1/2}(t) dt.
\end{equation}
%%%%%%%%%%%%%%%%%%%%%%%%%
Both variables obviously vanish at the turning point $z=z_{0}$, and $\zeta$ is an analytic function of $z$ at the this point, whereas $\xi$ has a branch point there. For convenience we choose the branch so that $\xi$ is  positive when $z \to z_{0}$ through values in which $\zeta$ is positive, with $\xi$ defined by continuity elsewhere.

For $j=0,\pm 1$ we denote suitable reference points in the $\zeta$ plane at infinity by $\zeta = \zeta^{(0)}$, where $\Re(\zeta^{(0)})= +\infty$, and $\Re (\zeta^{(\pm 1)} e^{\mp 2\pi i/3}) = +\infty$. We also denote $z=z^{(j)}$ that correspond to these, which are singularities of (\ref{eq05}), not necessarily at infinity. Let $\xi^{(j)}$ correspond to $\zeta^{(j)}$, and note that $\Re(\xi^{(0)})=+\infty$ and $\Re(\xi^{(\pm 1)})=-\infty$.

Next, let $Z$ be the $z$ domain containing $z=z_{0}$ in which $f(z)$ has no other zeros, and in which $f(z)$ and $g(z)$ are meromorphic, with poles (if any) at finite points $z=w_{j}$ ($j=1,2,3,\ldots $) such that at these singularities

(i) $f(z)$ has a pole of order $m>2$, and $g(z)$ is analytic or has a pole of order less than $\frac{1}{2}m+1$, or

(ii) $f(z)$ and $g(z)$ have a double pole, and 
$\left( z-w_{j}\right) ^{2}{g(z) \rightarrow -}\frac{1}{4}$ as $z\rightarrow w_{j}$.

We call these \emph{admissible poles}. The points $z^{(j) }$ must lie in $Z$. Further, these must be either at an admissible pole, or at infinity if $f(z)$ and $g(z)$ can be expanded in convergent series in a neighborhood of $z=\infty $ of the form
\begin{equation}
f(z)=z^{p}\sum\limits_{s=0}^{\infty }f_{s}z^{-s},\ g(z)=z^{q}\sum\limits_{s=0}^{\infty }g_{s}{z}^{-s},
\label{fginfinity}
\end{equation}
where $f_{0}\neq 0$, $g_{0}\neq 0$, and either $p$ and $q$ are integers such that $p>-2$ and $q<\frac{1}{2}p-1$, or $p=q=-2$ and $g_{0}=-\frac{1}{4}$. For details and generalizations of (\ref{fginfinity}) see \cite[Chap. 10, \S\S 4 and 5]{Olver:1997:ASF}.

We consider three numerically satisfactory solutions of the ODE (\ref{eq05}) denoted by $W_{j}(u,z)$ ($j=0,\pm 1$), these being recessive at $z=z^{(j)}$. As $u \to \infty$ these solutions possess the following Liouville-Green (LG) asymptotic expansions in exponential form
%%%%%%%%%%%%%%%%%%%
\begin{multline}
\label{eq11}
W_{0}(u,z) 
= \frac{1}{f^{1/4}(z)}
\exp \left\{ -u\xi
+\sum\limits_{s=1}^{N-1}{(-1) ^{s}
\frac{\hat{E}_{s}\left(
z\right) -\hat{E}_{s}
\left( {z^{(0) }}\right) }{u^{s}}}\right\} 
\\
\times \left\{1+\mathcal{O}\left(
\frac{1}{u^{N}}\right)\right\}
\quad (z \in Z_{0})
\end{multline}
%%%%%%%%%%%%%%%%%%%
and
%%%%%%%%%%%%%%%%%%%
\begin{multline}
\label{eq12}
W_{\pm 1}(u,z)
= \frac{1}{f^{1/4}(z)}
\exp \left\{ u\xi
+\sum\limits_{s=1}^{N-1}\frac{\hat{E}_{s}(z) 
-\hat{E}_{s}\left( {z^{(\pm 1) }}\right) }
{u^{s}}\right\} 
\\
\times \left\{1+\mathcal{O}\left(
\frac{1}{u^{N}}\right)\right\}
\quad (z \in Z_{\pm 1}).
\end{multline}
%%%%%%%%%%%%%%%%%%%
where $N \geq 2$ is an arbitrary integer. 

The coefficients $\hat{E}_{s}(z)$, are described in \cite[Eqs. (1.6), (1.12) - (1.15)]{Dunster:2021:SEB}, and bounds for the $\mathcal{O}$ terms are furnished by \cite[Thm. 2.1]{Dunster:2021:SEB}. For $j=0,\pm 1$ each LG region of validity $Z_{j} \subset Z$ contains the point $z^{(j)}$, and comprises the $z$ point set for which there is a \textit{progressive} path $L_{j}$ (say) linking $z$ with $z^{(j) }$ in $Z$ and having the properties (i) $L_{j} $ consists of a finite chain of $R_{2}$ arcs (as defined in \cite[Chap. 5, \S 3.3]{Olver:1997:ASF}), and (ii) as $v$ passes along $L_{j} $ from $z^{(j) }$ to $z$, the real part of $(-1)^{j}\xi(v)$ is nonincreasing,
where $\xi(v)$ is given by (\ref{eq06}) with $z=v$, and with the chosen sign fixed throughout.

An important property of the odd coefficients is that with suitable choice of integration constants each of $(z-z_{0})^{1/2}\hat{E}_{2s+1}(z)$ ($s=0,1,2,\ldots$) is meromorphic at $z=z_{0}$. We also emphasize that evaluation of these odd coefficients involves at most one integration of explicit functions, and the even ones do not require any integration (see \cite[Eq. (1.15)]{Dunster:2021:SEB}). This is in contrast to LG expansions in the more standard form, where the asymptotic series lies outside the exponential function; see \cite[Chap. 10]{Olver:1997:ASF}.

The three solutions are linearly dependent, and it is convenient to write their linear relationship in the form
%%%%%%%%%%%%%%%%%%%
\begin{equation}
\label{eq13}
iW_{0}(u,z)
=\lambda_{-1}(u)W_{-1}(u,z)
-\lambda_{1}(u)W_{1}(u,z).
\end{equation}
%%%%%%%%%%%%%%%%%%%
Usually the coefficients $\lambda_{\pm 1}(u)$ can be determined from the special functions in question, or if not, asymptotically via \cite[Lemma 2.3]{Dunster:2021:SEB}
%%%%%%%%%%%%%%%%%%%
\begin{equation}
\label{eq26a}
\lambda_{\pm 1}(u)
\sim   \exp \left\{ 
\sum\limits_{s=1}^{\infty}{
\frac{\hat{E}_{s}\left( {z^{(\pm 1) }}\right)
+(-1)^{s+1}\hat{E}_{s}
\left( {z^{(0) }}\right) }{u^{s}}}\right\} 
\quad (u \to \infty).
\end{equation}
%%%%%%%%%%%%%%%%%%%

Incidentally, there is an error in Eqs. (2.10) and (2.13) of that Lemma, namely $j=0,\pm 1$ should be replaced by $j=\pm 1$ in the former, and $\{j,k\} \in \{0,1,-1\}$ should be replaced by $\{j,k\} \in \{1,-1\}$ in the latter. The identities do not hold for $j,k=0$, and therefore the first sentence of the proof should also be discarded. This is a minor consideration, since $j,k=\pm 1$ is the non-trivial case, with $j,k=0$ of no interest.

The three LG expansions break down at the turning point $z_{0}$, and in order to obtain representations valid at this point the standard method involves the Airy function $\mathrm{Ai}(z)$ and its derivative \cite[Chap. 11]{Olver:1997:ASF}, along with two sets of asymptotic series involving inverse powers of the large parameter $u$. From \cite[Eqs.(3.1) - (3.3)]{Dunster:2021:SEB} alternative asymptotic forms for solutions are given in terms of slowly-varying coefficient functions $A(u,z)$ and $B(u,z)$ via the identities
%%%%%%%%%%%%%%%%%%%
\begin{equation}
\label{eq17}
\frac{1}{2\pi ^{1/2}u^{1/6}}W_{0}(u,z) 
=\mathrm{Ai}\left( {u^{2/3}\zeta }\right) A(u,z) 
+\mathrm{Ai}^{\prime }
\left( {u^{2/3}\zeta }\right) B(u,z),
\end{equation}
%%%%%%%%%%%%%%%%%%%
and 
%%%%%%%%%%%%%%%%%%%
\begin{equation}
\label{eq15}
\frac{e^{\pm \pi i/6}\lambda_{\pm 1}(u)}
{2\pi ^{1/2}u^{1/6}}W_{\pm 1}(u,z)
=\mathrm{Ai}_{\pm 1}\left( {u^{2/3}\zeta }\right) A(u,z)
+\mathrm{Ai}_{\pm 1}^{\prime }
\left( {u^{2/3}\zeta }\right) B(u,z),
\end{equation}
%%%%%%%%%%%%%%%%%%%
where $\mathrm{Ai}_{\pm 1}(z)=\mathrm{Ai}(ze^{\mp 2\pi i/3})$. An important feature is that these equations are linearly dependent, in that given any pair of these three identities the other one follows from (\ref{eq13}) and the Airy function connection formula \cite[Eq. 9.2.12]{NIST:DLMF}
%%%%%%%%%%%%%%%%%%%
\begin{equation}
\label{eq16}
i\mathrm{Ai}\left( {u^{2/3}\zeta }\right) 
=e^{\pi i/6}\mathrm{Ai}_{-1}\left( {u^{2/3}\zeta }
\right)
-e^{-\pi i/6}\mathrm{Ai}_{1}
\left( {u^{2/3}\zeta }\right).
\end{equation}
%%%%%%%%%%%%%%%%%%%
Indeed this is the main reason for the choice of constants multiplying $W_{j}(u,z)$ in the definitions (\ref{eq17}) and (\ref{eq15}), with the common factor $2 \pi ^{1/2}u^{1/6}$ in the denominators simply being a normalizing factor so that $A(u,z) \sim 1$ as $u \to \infty$ when $z \in Z_{j}$ ($j=0,\pm 1$).

The functions $A(u,z)$ and $B(u,z)$ can then be determined explicitly by solving simultaneously any pair of the three equations (\ref{eq17}) and (\ref{eq15}). One finds they are analytic in a domain containing the turning point $z_{0}$ (see \cite[\S 2]{Dunster:2021:SEB}). Most significantly, for large $u$ these two functions possess the readily computable asymptotic expansions given by \cite[Eqs. (1.20) and (1.21)]{Dunster:2021:SEB}, and these are valid for $z \in \mathbf{Z}$ where
%%%%%%%%%%%%%%%%%%%
\begin{equation}
\label{eq16a}
\mathbf{Z}=\left\{Z_{0} \cap Z_{1}\right\}
\cup \left\{Z_{0} \cap Z_{-1}\right\}
\cup \left\{Z_{1} \cap Z_{-1}\right\}.
\end{equation}
%%%%%%%%%%%%%%%%%%%
Recall that $Z_{j}$ ($j=0,\pm 1$) are the domains of validity of the LG expansions (\ref{eq11}) and (\ref{eq12}), and are described in the paragraph proceeding those equations.

Note this region surrounds the turning point $z_{0}$, but does not include it since the LG expansions break down at that point. However, asymptotic expansions for $A(u,z)$ and $B(u,z)$ can be computed in a neighborhood containing the turning point via Cauchy's integral formula (\cite[Thm. 4.2]{Dunster:2021:SEB}).

For the same functions defined by the unique asymptotic behavior (\ref{eq11}) and (\ref{eq12}), we follow \cite{Cherry:1950:UAF} and \cite{Dunster:2014:OEB} and seek solutions that only involve one Airy function and not its derivative. Specifically, in place of (\ref{eq17}) and (\ref{eq15}) we instead \emph{define} functions $\mathcal{Y}(u,z)$ and $\mathcal{Z}(u,z)$ via the pair of equations
%%%%%%%%%%%%%%%%%%%
\begin{equation}
\label{eq18}
e^{\pm \pi i/6}\lambda_{\pm 1}(u)W_{\pm 1}(u,z)
=c(u)\mathcal{Y}(u,z)
\mathrm{Ai}_{\pm 1}
\left( u^{2/3}\mathcal{Z}(u,z)\right),
\end{equation}
%%%%%%%%%%%%%%%%%%%
%%%%%%%%%%%%%%%%%%%
% \begin{equation}
% \label{eq18}
% \frac{e^{\pm \pi i/6}\lambda_{\pm 1}(u)}
% {2\pi ^{1/2}u^{1/6}}W_{\pm 1}(u,z)
% =c(u)\mathcal{Y}(u,z)
% \mathrm{Ai}_{\pm 1}
% \left( u^{2/3}\mathcal{Z}(u,z)\right)
% \end{equation}
%%%%%%%%%%%%%%%%%%%
where $c(u) \neq 0$ is arbitrary. Then using the connection formulas (\ref{eq13}) and (\ref{eq16}) we get
%%%%%%%%%%%%%%%%%%%
\begin{equation}
\label{eq19}
W_{0}(u,z)
=c(u)\mathcal{Y}(u,z) \mathrm{Ai}
\left( u^{2/3}\mathcal{Z}(u,z)\right),
\end{equation}
%%%%%%%%%%%%%%%%%%%
% %%%%%%%%%%%%%%%%%%%
% \begin{equation}
% \label{eq19}
% \frac{1}
% {2\pi ^{1/2}u^{1/6}}W_{0}(u,z)
% =c(u)\mathcal{Y}(u,z) \mathrm{Ai}
% \left( u^{2/3}\mathcal{Z}(u,z)\right)
% \end{equation}
%%%%%%%%%%%%%%%%%%%
which as we shall see renders $c(u)\mathcal{Y}(u,z)$ and $\mathcal{Z}(u,z)$ slowly varying as functions of large $u$ in a full neighborhood of $z=z_{0}$ (similarly to $A(u,z)$ and $B(u,z)$). We also demonstrate that they are analytic at this point.

From \cite{Dunster:2014:OEB} one expects that as $u \to \infty$
%%%%%%%%%%%%%%%%%%%
\begin{equation}
\label{eq32}
\mathcal{Z}(u,z) \sim \zeta
+\sum_{s=1}^{\infty}
\frac{\Upsilon_{s}(z)}{u^{2s}},
\end{equation}
%%%%%%%%%%%%%%%%%%%
uniformly for $z \in \mathbf{Z} \cup \{z_{0}\}$, where $\Upsilon_{s}(z)$ ($s=1,2,3,\ldots$) are analytic in this region. We show that this is indeed true, and while obtaining simpler expressions for the coefficients $\Upsilon_{s}(z)$ than the ones given in \cite{Dunster:2014:OEB}. To this end, on eliminating $\mathcal{Y}(u,z)$ we have the following three implicit equations for $\mathcal{Z}(u,z)$
%%%%%%%%%%%%%%%%%%%
\begin{equation}
\label{eq20}
\frac{e^{\pm \pi i/6}\lambda_{\pm 1}(u)
W_{\pm 1}(u,z)}
{W_{0}(u,z)}
=\frac{\mathrm{Ai}_{\pm 1}
\left( u^{2/3}\mathcal{Z}(u,z)\right)}
{\mathrm{Ai}\left( u^{2/3}\mathcal{Z}(u,z)\right)},
\end{equation}
%%%%%%%%%%%%%%%%%%%
and 
%%%%%%%%%%%%%%%%%%%
\begin{equation}
\label{eq21}
\frac{e^{-\pi i/3}
\lambda_{-1}(u) W_{-1}(u,z)}
{\lambda_{1}(u) W_{1}(u,z)}
=\frac{\mathrm{Ai}_{-1}
\left( u^{2/3}\mathcal{Z}(u,z)\right)}
{\mathrm{Ai}_{1}\left( u^{2/3}\mathcal{Z}(u,z)\right)}.
\end{equation}
%%%%%%%%%%%%%%%%%%%

We shall select one of these three relations in turn to derive an asymptotic expansion for $\mathcal{Z}(u,z)$ for $z \in \mathbf{Z}$, the choice of which depends on the region where, by linear independence of any given pair, one and only one of the functions $W_{j}(u,z)$ ($j=0,\pm 1$) appearing in the LHS of (\ref{eq20}) or (\ref{eq21}) is recessive. Note that for each point in $\mathbf{Z}$ there are two such choices.

Before we proceed on this, we prove the following important result.

\begin{lemma}
\label{lemma01}
$\mathcal{Z}(u,z)$ is analytic at points, which includes the turning point $z=z_{0}$, where all three of $W_{j}(u,z)$ ($j=0,\pm 1$) are analytic.
\end{lemma}

\begin{proof}
At such a point choose $j,k \in \{0,\pm 1\}$, $j \neq k$, such that neither $W_{k}(u,z)$ nor $W_{j}(u,z)$ is zero at the point, and then consider the one of (\ref{eq20}) or (\ref{eq21}) that corresponds to these values of $j$ and $k$. The corresponding Airy functions also then do not vanish at this point. Analyticity of $\mathcal{Z}(u,z)$ at the point in question is then a consequence of applying the inverse function theorem for analytic functions (see for example \cite[\S 6.22]{Copson:1935:ITF} and \cite[Thm. 2.1]{Fabijonas:1999:ORA}), noting that for $j=0,\pm 1$ $\mathrm{Ai}_{j}(t)$ are entire, and for $j,k=0,\pm 1$, $j \neq k$
%%%%%%%%%%%%%%%%%%%
\begin{equation}
\label{eq22}
\frac{d}{dt} \left\{
\frac{\mathrm{Ai}_{j}(t)}
{\mathrm{Ai}_{k}(t)}\right\}
=-\frac{ \mathscr{W}\left\{\mathrm{Ai}_{j}(t)
\mathrm{Ai}_{k}(t)\right\}}{\mathrm{Ai}_{k}^{2}(t)}
\neq 0;
\end{equation}
%%%%%%%%%%%%%%%%%%%
see \cite[\S 9.2(iv)]{NIST:DLMF}.

Finally $z_{0}$ is an ordinary point of the ODE (\ref{eq05}), and as such its solutions $W_{j}(u,z)$ are analytic at this point.
\end{proof}

Next let $z \in \mathbf{Z}$ such that $|\arg(\mathcal{Z}) |\leq \pi -\delta$ ($\delta>0$). Then from \cite[Eq. (A.3)]{Dunster:2017:COA} 
%%%%%%%%%%%%%%%%%%%
\begin{equation}
\label{eq25}
\mathrm{Ai}\left( u^{2/3}\mathcal{Z}\right) 
\sim \frac{1}{2\pi ^{1/2}u^{1/6}\mathcal{Z}
^{1/4}}\exp \left\{ -u \mathcal{X}
+\sum\limits_{s=1}^{\infty}{(-1) ^{s}
\frac{a_{s}}{s(u \mathcal{X})^{s}}}\right\}
\quad (u \mathcal{X} \to \infty),
\end{equation}
%%%%%%%%%%%%%%%%%%%
where
%%%%%%%%%%%%%%%%%%%
\begin{equation}
\label{eq28}
\mathcal{X} =\mathcal{X}(u,z) =
\tfrac{2}{3}\mathcal{Z}^{3/2}(u,z),
\end{equation}
%%%%%%%%%%%%%%%%%%%
and $\{a_{s}\}_{s=1}^{\infty}$ is the sequence of rational numbers given recursively by
%%%%%%%%%%%%%%%%%%%
\begin{equation}  
\label{eq24}
a_{1}=a_{2}=\frac{5}{72}, \quad
a_{s+1}=\frac{1}{2}\left(s+1\right) a_{s}
+\frac{1}{2}
\sum\limits_{j=1}^{s-1}{a_{j}a_{s-j}} 
\quad (s=2,3,4,\ldots).
\end{equation}
%%%%%%%%%%%%%%%%%%%
Hence for $|\arg(\mathcal{Z} e^{\mp 2\pi i/3})| \leq \pi -\delta$ ($\delta>0$)
%%%%%%%%%%%%%%%%%%%
\begin{equation}
\label{eq26}
\mathrm{Ai}_{\pm 1}\left( {u^{2/3}\mathcal{Z} }\right) 
\sim  \frac{e^{\pm \pi i/6}}{2\pi^{1/2}u^{1/6}
\mathcal{Z}^{1/4}}\exp \left\{ u\mathcal{X}
+\sum\limits_{s=1}^{\infty}\frac{a_{s}}
{s (u \mathcal{X})^{s}}\right\}
\quad (u \mathcal{X} \to \infty).
\end{equation}
%%%%%%%%%%%%%%%%%%%

Consider the RHS of (\ref{eq20}), and let $z \in \mathbf{Z}$ such that $|\arg(\mathcal{Z} \exp(\mp \frac13 \pi i)) |\leq \frac23\pi -\delta$ ($\delta>0$). Then from (\ref{eq25}) and (\ref{eq26}) it follows that as $|u \mathcal{X}| \to \infty$
%%%%%%%%%%%%%%%%%%%
\begin{equation}
\label{eq27a}
\frac{\mathrm{Ai}_{\pm 1}
\left( u^{2/3}\mathcal{Z}(u,z)\right)}
{\mathrm{Ai}\left( u^{2/3}\mathcal{Z}(u,z)\right)}
\sim  
\exp \left\{2 u\mathcal{X}(u,z)
+2\sum\limits_{s=0}^{\infty}\frac{
a_{2s+1}}{(2s+1) \left\{u 
\mathcal{X}(u,z)\right\}^{2s+1}}\right\}.
\end{equation}
%%%%%%%%%%%%%%%%%%%

Now this is either exponentially small or exponentially large, according to the sign of $\Re\{\mathcal{X}(u,z)\}$, or oscillatory if $\Re\{\mathcal{X}(u,z)\}=0$. It follows that the same is true for the LHS of (\ref{eq20}) for our chosen values of $z$, and hence for the upper sign case one of $W_{1}(u,z)$ and $W_{0}(u,z)$ must be recessive (with the other obviously then dominant). We can discard the possibility of both being dominant, for otherwise their ratio would approach a finite limit.

The same is true for $W_{-1}(u,z)$ and $W_{0}(u,z)$ in the lower sign case. Since $z \in \mathbf{Z}$ the former corresponds to $z \in Z_{0} \cap Z_{1}$ and the latter holds for $z \in Z_{0} \cap Z_{-1}$. We have thus proven that for $z \in Z_{0} \cap Z_{1}$ or $z \in Z_{0} \cap Z_{-1}$, on invoking (\ref{eq11}), (\ref{eq12}), (\ref{eq26a}), (\ref{eq20}) and (\ref{eq27a}), that
%%%%%%%%%%%%%%%%%%%
\begin{multline}
\label{eq27}
\exp \left\{ 2 u\xi
+2\sum\limits_{s=0}^{\infty}
\frac{\hat{E}_{2s+1}(z)}
{u^{2s+1}}\right\} \\
\sim  
\exp \left\{2 u\mathcal{X}(u,z)
+2\sum\limits_{s=0}^{\infty}\frac{
a_{2s+1}}{(2s+1) \left\{u 
\mathcal{X}(u,z)\right\}^{2s+1}}\right\}
\quad (u\mathcal{X} \to \infty).
\end{multline}
%%%%%%%%%%%%%%%%%%%

It remains to consider $z \in Z_{1} \cap Z_{-1}$. For this case we use (\ref{eq12}), (\ref{eq21}) and (\ref{eq26}), but in these we must take into account that the branch of $\xi$ for the asymptotic expansions of $W_{1}(u,z)$ is of opposite sign than those for $W_{-1}(u,z)$, and similarly for $\mathcal{X}$ for the expansion of $\mathrm{Ai}_{1}(u^{2/3}\mathcal{Z})$. For example, suppose $\frac{2}{3}\pi\leq \arg(\mathcal{Z}) \leq \pi$ ($\pi  \leq \arg(\mathcal{X}) \leq \frac{3}{2}\pi$). Then (\ref{eq26}) holds for $\mathrm{Ai}_{ 1}(u^{2/3}\mathcal{Z})$ but not for $\mathrm{Ai}_{-1}(u^{2/3}\mathcal{X})$.

To obtain the corresponding expansion for $\mathrm{Ai}_{-1}(u^{2/3}\mathcal{Z} )$ we can replace $\mathcal{Z}$ with $\mathcal{Z} e^{-2\pi i}$ in (\ref{eq26}), and consequently $\mathcal{X}$ by $-\mathcal{X}$. Then as $u\mathcal{X} \to \infty$ we have for (at least) $\frac{2}{3}\pi \leq \arg(\mathcal{Z}) \leq \pi$
%%%%%%%%%%%%%%%%%%%
\begin{equation}
\label{eq29}
\mathrm{Ai}_{-1}\left(u^{2/3}\mathcal{Z}\right) 
\sim 
\frac{i e^{-\pi i/6}}{2\pi^{1/2}u^{1/6}\mathcal{Z}^{1/4}}
\exp \left\{ -u\mathcal{X}
+\sum\limits_{s=1}^{\infty}{(-1) ^{s}
\frac{a_{s}}{s (u\mathcal{X})^{s}}}\right\} 
\quad (u\mathcal{X} \to \infty).
\end{equation}
%%%%%%%%%%%%%%%%%%%

Since $\Re(\mathcal{X}) \leq 0$ in the sector under consideration, it should be noted that $\mathrm{Ai}_{1}(u^{2/3}\mathcal{Z})$ is recessive and $\mathrm{Ai}_{-1}(u^{2/3}\mathcal{Z} )$ is dominant as $u^{2/3}\mathcal{Z} \to \infty$.

We remark that (\ref{eq29}) could also be deduced from (\ref{eq16}), (\ref{eq25}) and (\ref{eq26}) (with upper signs taken), but using these formulas it only could be shown to be true for $\frac{2}{3}\pi \leq \arg(\mathcal{Z}) \leq \pi- \delta $ ($\delta > 0$), since the expansion (\ref{eq25}) breaks down at $\arg(\mathcal{Z}) = \pi$ (where its zeros are located).

Similarly, for $z$ lying in the subdomain of $Z_{1} \cap Z_{-1}$ where $\pi \leq \arg(\xi) \leq \frac{3}{2}\pi$ we find in place of (\ref{eq12}) (cf. the proof of \cite[Lemma 2.3]{Dunster:2021:SEB})
%%%%%%%%%%%%%%%%%%%
\begin{equation}
\label{eq12a}
W_{-1}(u,z) \sim
\frac{i}{f^{1/4}(z)}
\exp \left\{ -u\xi
+\sum\limits_{s=1}^{\infty}\frac{(-1)^{s}\hat{E}_{s}(z) 
-\hat{E}_{s}\left( {z^{(-1) }}\right) }
{u^{s}}\right\}
\quad  (u\xi \to \infty),
\end{equation}
%%%%%%%%%%%%%%%%%%%
with (\ref{eq12}) still holding for $W_{1}(u,z)$. Note in (\ref{eq12a}) we have $\Re(\xi) \leq 0$ and hence this function is dominant in the region in question, and from (\ref{eq12}) $W_{1}(u,z)$ is recessive.

Then from (\ref{eq12}), (\ref{eq26a}), (\ref{eq21}), (\ref{eq26}), (\ref{eq29}) and (\ref{eq12a}) we have that (\ref{eq27}) holds for $z \in Z_{1} \cap Z_{-1}$ such that $\pi \leq \arg(\xi) \leq \frac{3}{2}\pi$. Likewise one can show that the same is also true for points in this domain where $-\frac{3}{2}\pi \leq \arg(\xi) \leq -\pi$, and this establishes that the asymptotic expansion (\ref{eq27}) holds for $z \in \mathbf{Z}$ as defined by (\ref{eq16a}).

So from (\ref{eq27}) we have derived the asymptotic expansion
%%%%%%%%%%%%%%%%%%%
\begin{equation}
\label{eq31}
u\left\{\mathcal{X}(u,z)-\xi \right\}
+\sum\limits_{s=0}^{\infty}\frac{
a_{2s+1}}{(2s+1) \left\{u 
\mathcal{X}(u,z)\right\}^{2s+1}}
-\sum\limits_{s=0}^{\infty}
\frac{\hat{E}_{2s+1}(z)}
{u^{2s+1}}
\sim  0
\quad (u \to \infty),
\end{equation}
%%%%%%%%%%%%%%%%%%%
which is uniformly valid for $z \in \mathbf{Z}$ with $|\mathcal{X}(u,z)| \geq \delta >0$. The asymptotic relation (\ref{eq31}) establishes that for the stated values of $z$ there exists an asymptotic expansion of the form $\mathcal{X}(u,z) \sim \sum_{s=0}^{\infty}\Theta_{s}(z) u^{-2s}$ as $u \to \infty$. From (\ref{eq31}) we perceive that the leading term is clearly $\Theta_{0} = \xi$, and thus from (\ref{eq28}), and recalling that $\xi = \frac23 \zeta^{3/2}$,
%%%%%%%%%%%%%%%%%%%
\begin{equation}
\label{eq31a}
\mathcal{Z}(u,z) \sim \zeta
\left\{1+\sum_{s=1}^{\infty}
\frac{3\Theta_{s}(z)}{2 u^{2s} \xi}\right\}^{2/3}.
\end{equation}
%%%%%%%%%%%%%%%%%%%
On expanding this via \cite[Chap. 1, Ex. 8.4]{Olver:1997:ASF} we obtain the asserted asymptotic series of the form (\ref{eq32}). 

Returning to (\ref{eq31}), we insert in this the expansion for $\mathcal{X}(u,z)$, expand as an asymptotic series in inverse powers of $u$, and set the coefficient of $u^{-1}$ to zero, noting that the leading $u$ terms cancel since $\Theta_{0}(z)=\xi$. As a result we find $\Theta_{1}(z)$ explicitly, and from this, (\ref{eq32}) and (\ref{eq31a}) yields
%%%%%%%%%%%%%%%%%%%
\begin{equation}
\label{eq44a}
\Upsilon_{1}(z)=
\frac{3 \xi \hat{E}_{1}(z)}{2\zeta^{2}}
-\frac{5}{48\zeta^{2}}.
\end{equation}
%%%%%%%%%%%%%%%%%%%

To find an recursion for the other coefficients $\Upsilon_{s}(z)$ ($s=2,3,4,\ldots$), we have from (\ref{eq06}), (\ref{eq32}) and (\ref{eq28}) as $u \to \infty$
%%%%%%%%%%%%%%%%%%%
\begin{equation}
\label{eq33}
\mathcal{X}(u,z) \sim \xi
\left\{1+\sum_{s=1}^{\infty}
\frac{\Upsilon_{s}(z)}{u^{2s}\zeta}\right\}^{3/2},
\end{equation}
%%%%%%%%%%%%%%%%%%%
and hence for $s=-1,0,1,2,\ldots$
%%%%%%%%%%%%%%%%%%%
\begin{equation}
\label{eq34}
\frac{1}{\mathcal{X}(u,z)^{2s+1}}
\sim 
\frac{1}{\xi^{2s+1}}
\left\{1
+\sum_{l=1}^{\infty}
\frac{\Upsilon_{l}(z)}{u^{2l} \zeta}
\right\}^{-3s-(3/2)}.
\end{equation}
%%%%%%%%%%%%%%%%%%%
Now again use \cite[Chap. 1, Ex. 8.4]{Olver:1997:ASF} to obtain
%%%%%%%%%%%%%%%%%%%
\begin{equation}
\label{eq35}
\frac{1}{\left\{u 
\mathcal{X}(u,z)\right\}^{2s+1}}
\sim 
\frac{1}{\xi^{2s+1}}
\sum_{l=0}^{\infty} 
\frac{d_{s,l}(z)}{u^{2l+2s+1}},
\end{equation}
%%%%%%%%%%%%%%%%%%%
where $d_{s,0}(z)=1$ and for $l=1,2,3,\ldots$, and any integer $s$
%%%%%%%%%%%%%%%%%%%
\begin{equation}
\label{eq36}
d_{s,l}(z)
=-{\frac{1}{2 l \zeta}}
\sum _{k=1}^{l}
\left\{(6s+1)k+2l \right\}
\Upsilon_{k}
\left( z \right) d_{s,l-k}(z).
\end{equation}
%%%%%%%%%%%%%%%%%%%
So from (\ref{eq35}) as $u \to \infty$
%%%%%%%%%%%%%%%%%%%
\begin{equation}
\label{eq37}
\sum\limits_{s=0}^{\infty}\frac{
a_{2s+1}}{(2s+1) \left\{u 
\mathcal{X}(u,z)\right\}^{2s+1}}
\sim  
\sum\limits_{s=0}^{\infty}
\frac{g_{s}(z)}{u^{2s+1}},
\end{equation}
%%%%%%%%%%%%%%%%%%%
where
%%%%%%%%%%%%%%%%%%%
\begin{equation}
\label{eq38}
g_{s}(z)=\sum_{k=0}^{s}\frac{a_{2k+1}d_{k,s-k}(z)}{(2k+1)\xi^{2k+1}}.
\end{equation}
%%%%%%%%%%%%%%%%%%%

For example we find that
%%%%%%%%%%%%%%%%%%%
\begin{equation}
\label{eq39}
g_{0}(z)=\frac{5}{72 \xi},
\end{equation}
%%%%%%%%%%%%%%%%%%%
and
%%%%%%%%%%%%%%%%%%%
\begin{equation}
\label{eq40}
g_{1}(z)=\frac{5}
{31104\xi^{3}}
\left\{221- 288 \zeta^{2} 
\Upsilon_{1}(z)\right\}.
\end{equation}
%%%%%%%%%%%%%%%%%%%
Then using (\ref{eq31}), (\ref{eq35}) with $s=-1$, and (\ref{eq37}) we have by setting each coefficient of $u^{-2s-1}$ to zero
%%%%%%%%%%%%%%%%%%%
\begin{equation}
\label{eq41}
\xi d_{-1,s+1}(z) +g_{s}(z)
-\hat{E}_{2s+1}(z)=0
\quad (s=0,1,2,\ldots).
\end{equation}
%%%%%%%%%%%%%%%%%%%

Note from (\ref{eq36}) for $s=0,1,2,\ldots$
%%%%%%%%%%%%%%%%%%%
\begin{equation}
\label{eq42}
d_{-1,s+1}(z)
=-\frac{1}{2 (s+1) \zeta}
\sum _{k=1}^{s+1}
\left\{2s+2-5k \right\}
\Upsilon_{k}(z) d_{-1,s-k+1}(z),
\end{equation}
%%%%%%%%%%%%%%%%%%%
in which $d_{-1,0}(z)=1$. Thus 
%%%%%%%%%%%%%%%%%%%
\begin{equation}
\label{eq43}
d_{-1,1}(z)
=\frac{3 \Upsilon_{1}(z)}{2\zeta},
\end{equation}
%%%%%%%%%%%%%%%%%%%
which appears in the $s=0$ term of (\ref{eq41}). With this and (\ref{eq39}) we confirm the leading coefficient (\ref{eq44a}).

Each $\Upsilon_{s}(z)$ ($s=2,3,4,\dots$) can then similarly be obtained in turn by solving (\ref{eq41}), with these other coefficients given in terms of the previous ones. The general recursion formula is readily found to be as given by (\ref{eq43a}) below. From this, the next three coefficients are found to be
%%%%%%%%%%%%%%%%%%%
\begin{equation}
\label{eq45}
\Upsilon_{2}(z)=
-\frac{\Upsilon_{1}^{2}(z)}{4 \zeta}
+\frac{5 \Upsilon_{1}(z)}{32\zeta^{3}}
+\frac{3 \xi \hat{E}_{3}(z)}{2\zeta^{2}}
-\frac{1105}{9216\zeta^{5}},
\end{equation}
%%%%%%%%%%%%%%%%%%%
%%%%%%%%%%%%%%%%%%%
\begin{multline}
\label{eq46}
\Upsilon_{3}(z)=
-\frac{\Upsilon_{1}(z)\Upsilon_{2}(z)}{2 \zeta}
+\frac{\Upsilon_{1}^{3}(z)}{24 \zeta^{2}}
-\frac{25\Upsilon_{1}^{2}(z)}{128 \zeta^{4}}
+\frac{5 \Upsilon_{2}(z)}{32\zeta^{3}}
\\
+\frac{1105 \Upsilon_{1}(z)}{2048 \zeta^{6}}
+\frac{3 \xi \hat{E}_{5}(z)}{2 \zeta^2}
-\frac{82825}{98304 \zeta^{8}},
\end{multline}
%%%%%%%%%%%%%%%%%%%
and
%%%%%%%%%%%%%%%%%%%
\begin{multline}
\label{eq47}
\Upsilon_{4}(z)=
-\frac{\Upsilon_{1}^{4}(z)}{64  \zeta^{3}}
+\frac{\Upsilon_{1}^{2}(z)
\Upsilon_{2}(z)}{8\zeta^{2}}
+\frac{175 \Upsilon_{1}^{3}(z)}{768 \zeta^{5}}
-\frac{\Upsilon_{1}(z)\Upsilon_{3}(z)}{2\zeta}
\\
-\frac{25\Upsilon_{1}(z)\Upsilon_{2}(z)}{64\zeta^{4}}
-\frac{\Upsilon_{2}^{2}(z)}{4\zeta}
-\frac{12155 \Upsilon_{1}^{2}(z)}
{8192  \zeta^{7}}
+\frac{5\Upsilon_{3}(z)}{32\zeta^{3}}
\\
+\frac{1105\Upsilon_{2}(z)}{2048\zeta^{6}}
+\frac{414125\Upsilon_{1}(z)}{65536\zeta^{9}}
+\frac{3\xi\hat{E}_{7}(z)}{2\zeta^{2}}
-\frac{1282031525}{88080384\zeta^{11}}.
\end{multline}
%%%%%%%%%%%%%%%%%%%

Our expansion for $\mathcal{Z}(u,z)$ was the main priority, as we only need this for the expansions for the zeros. But for completeness we note that $\mathcal{Y}(u,z)$ can be expressed in terms of the derivative of $\mathcal{Z}(u,z)$ as follows\footnote{This method was suggested by one of the reviewers}. First on differentiating (\ref{eq20}) and using the Wronskian for Airy functions \cite[Eq. 9.2.8]{NIST:DLMF}
%%%%%%%%%%%%%%%%%%%
\begin{equation}
\label{eq48a}
\frac{\lambda_{\pm 1}(u)\mathscr{W}
\left\{W_{0}(u,z),W_{\pm 1}(u,z)\right\}}
{W_{0}^{2}(u,z)}
=\frac{u^{2/3}}{2 \pi \mathrm{Ai}^{2}(u^{2/3}
\mathcal{Z}(u,z))}
\frac{\partial \mathcal{Z}(u,z)}{\partial z}.
\end{equation}
%%%%%%%%%%%%%%%%%%%
Then substitute (\ref{eq19}) and we arrive at
%%%%%%%%%%%%%%%%%%%
\begin{equation}
\label{eq48b}
\mathcal{Y}(u,z)
=c_{1}(u)
\left\{\frac{\partial \mathcal{Z}(u,z)}
{\partial z}\right\}^{-1/2},
\end{equation}
%%%%%%%%%%%%%%%%%%%
where
%%%%%%%%%%%%%%%%%%%
\begin{equation}
\label{eq48c}
c_{1}(u)
=\frac{\left[2 \pi \lambda_{\pm 1}(u)
\mathscr{W}
\left\{W_{0}(u,z),W_{\pm 1}(u,z)\right\}\right]^{1/2}}
{u^{1/3}c(u)},
\end{equation}
%%%%%%%%%%%%%%%%%%%
which is independent of $z$ by virtue of Abel's identity.

The main result of this section reads as follows.
\begin{theorem}
\label{thm2.2}
Let $f(z)$ and $g(z)$ be meromorphic functions in a domain $Z$ satisfying the conditions of this section, with $z_{0} \in Z$ being a simple zero of $f(z)$. Then Eq. (\ref{eq05}), which has a simple turning point at $z=z_{0}$, has numerically satisfactory solutions $W_{j}(u,z)$ ($j=0,\pm 1$) given by (\ref{eq18}) and (\ref{eq19}) where $c(u) \neq 0$ is arbitrary, and $\lambda_{\pm 1}(u)$ are the connection coefficients in (\ref{eq13}), having the asymptotic expansions given by (\ref{eq26a}). In these expressions $c(u)\mathcal{Y}(u,z)$ and $\mathcal{Z}(u,z)$ are analytic for $z \in \mathbf{Z} \cup \{z_{0}\}$, where $\mathbf{Z}$ is defined by (\ref{eq16a}), and slowly-varying functions for large $u$ in this domain. In particular, $\mathcal{Z}(u,z)$ possesses the asymptotic expansion 
%%%%%%%%%%%%%%%%%%%
\begin{equation}
\label{eq32a}
\mathcal{Z}(u,z) \sim \zeta
+\sum_{s=1}^{\infty}
\frac{\Upsilon_{s}(z)}{u^{2s}}
\quad (u \to \infty),
\end{equation}
%%%%%%%%%%%%%%%%%%%
uniformly for $z\in\mathbf{Z} \cup \{z_{0}\}$, where
%%%%%%%%%%%%%%%%%%%
\begin{equation}
\label{eq44}
\Upsilon_{1}(z)=
\frac{3 \xi \hat{E}_{1}(z)}{2\zeta^{2}}
-\frac{5}{48\zeta^{2}},
\end{equation}
%%%%%%%%%%%%%%%%%%%
and for $s=1,2,3,\ldots$
%%%%%%%%%%%%%%%%%%%
\begin{multline}
\label{eq43a}
\Upsilon_{s+1}(z)
=\frac{1}{3(s+1)}
\sum_{k=1}^{s}
\left\{2s+2-5k \right\}
\Upsilon_{k}(z) d_{-1,s-k+1}(z)
\\
+\frac{3\xi}{2\zeta^{2}}\left\{\hat{E}_{2s+1}(z)-g_{s}(z)\right\}.
\end{multline}
%%%%%%%%%%%%%%%%%%%
Here, $\xi$ and $\zeta$ are given by (\ref{eq06}), $\hat{E}_{s}(z)$ ($s=1,2,3,\ldots$) are the LG coefficients in (\ref{eq11}) and (\ref{eq12}) given by \cite[Eqs. (1.6), (1.12) - (1.15)]{Dunster:2021:SEB}, $d_{s,l}(z)$ ($l=1,2,3,\ldots$ and $s$ any integer) are given recursively by (\ref{eq36}) starting with $d_{s,0}(z)=1$, and $g_{s}(z)$ ($s=0,1,2,\ldots$) are given by (\ref{eq38}). Moreover, the coefficients $\Upsilon_{s}(z)$ have a removable singularity at $\xi=\zeta=0$ (which correspond to $z=z_{0}$).
\end{theorem}

% %%%%%%%%%%%%%%%%%%%
% \begin{equation}
% \label{eq44a}
% W_{0}(u,z)
% =c(u)\mathcal{Y}(u,z) \mathrm{Ai}
% \left( u^{2/3}\mathcal{Z}(u,z)\right)
% \end{equation}
% %%%%%%%%%%%%%%%%%%%

% and

% %%%%%%%%%%%%%%%%%%%
% \begin{equation}
% \label{eq44b}
% W_{\pm 1}(u,z)
% =\frac{e^{\mp \pi i/6}c(u)}
% {\lambda_{\pm 1}(u)}
% \mathcal{Y}(u,z)
% \mathrm{Ai}_{\pm 1}
% \left( u^{2/3}\mathcal{Z}(u,z)\right)
% \end{equation}
% %%%%%%%%%%%%%%%%%%%

\begin{proof}
From (\ref{eq06}) we have $\xi^{2}=\frac49\zeta^{3}$. Thus from (\ref{eq41})
%%%%%%%%%%%%%%%%%%%
\begin{equation} 
\label{eq43c}
\tfrac49\zeta^{3} d_{-1,s+1}(z) +\xi g_{s}(z)
-\xi \hat{E}_{2s+1}(z) = 0
\quad (s=0,1,2,\ldots).
\end{equation}
%%%%%%%%%%%%%%%%%%%
Then from (\ref{eq38}) and (\ref{eq42}) one can readily prove by induction that each $\Upsilon_{s}(z)$ ($s=1,2,3,\ldots$) is meromorphic at $z=z_{0}$ ($\zeta = 0$), on noting that this is also true for $\xi \hat{E}_{2s+1}(z)$, and recalling that $\zeta$ is an analytic function of $z$ at this point. Thus the coefficients $\Upsilon_{s}(z)$ actually have a removable singularity at $z=z_{0}$ by virtue of \cite[Thm. 3.1]{Dunster:2021:NKF}. Moreover that cited theorem shows that the expansion (\ref{eq32}) is valid in a full neighborhood of the turning point $z=z_{0}$. Error bounds given in \cite[Thms. 2.1 and 2.4]{Dunster:2021:SEB} for the Airy and LG expansions (\ref{eq11}), (\ref{eq12}), (\ref{eq25}) and (\ref{eq26}) establishes that the same is true away from the turning point (specifically, at all points where these parent expansions are valid, namely $\mathbf{Z}$).
\end{proof}

For example, if near the turning point $f(z)=\sum_{s=1}^{\infty}f_{s}(z-z_{0})^{s}$ ($f_{1} \neq 0$) and $g(z)=\sum_{s=0}^{\infty}g_{s}(z-z_{0})^{s}$ then one finds from \cite[Eqs. (1.6), (1.12) and (1.14)]{Dunster:2021:SEB}, together with (\ref{eq06}) and (\ref{eq44a}), that
%%%%%%%%%%%%%%%%%%%
\begin{equation} 
\label{eq48}
\Upsilon_{1}(z)
=\tfrac{3}{140}12^{1/3}f_{1}^{-8/3}
\left\{
35f_{{1}}^{2}g_{0}+15f_{1}f_{3}-9f_{{2}}^{2}
\right\} + \mathcal{O}\left(z-z_{0}\right)
\quad (z \to z_{0}).
\end{equation}
%%%%%%%%%%%%%%%%%%%

\section{Zeros of Bessel functions}
\label{sec3}

The first step is to apply the Liouville transformations described in \cref{sec2} to Bessel's equation. To this end, we first note from \cite[Chap. 11, Eq. (10.01)]{Olver:1997:ASF} that the functions $w=z^{1/2}J_{\nu}(\nu z)$, $w=z^{1/2}H_{\nu}^{(1)}(\nu z)$ and $w=z^{1/2}H_{\nu}^{(2)}(\nu z)$ satisfy (\ref{eq05}) with $u=\nu$ and
%%%%%%%%%%%%%%%%%%%
\begin{equation} 
\label{eq50}
f(z)=\frac{1-z^{2}}{z^{2}},
\quad  g(z)=-\frac{1}{4z^{2}}.
\end{equation}
%%%%%%%%%%%%%%%%%%%
This ODE has turning points at $z=\pm 1$, and our expansions are valid in domains which contain the turning point $z=z_{0}=1$, but not the turning point $z=-1$. Note that the singularity at $z=0$ is admissible according to the definition given in \cref{sec2}, and $f(z)$ and $g(z)$ have the appropriate behavior (\ref{fginfinity}) at $z=\infty$, with $p=0$ and $q=-2$ in the notation of that condition.

Now from (\ref{eq06}) the appropriate Liouville transformation for $z \neq 0$ is 
%%%%%%%%%%%%%%%%
\begin{equation}
\label{eq51}
\xi =\frac{2}{3}\zeta ^{3/2}
= \int_{z}^{1}\frac{\left(1-t^2\right)^{1/2}}{t}dt
=\ln \left\{ {\frac{1+\left( {1-z^{2}}\right)^{1/2}}
{z}}\right\}
-\left( {1-z^{2}}\right) ^{1/2}.
\end{equation}
%%%%%%%%%%%%%%%%%%
The transformed variable $\zeta$ is real for positive $z$, with $\zeta \to \pm \infty$ as $z \to 0,+\infty$. It is analytic at $z=1$ and can be defined by analytic continuation in the whole complex plane cut along the negative real axis. $\xi$ is positive for $z\in (0,1)$ and defined by continuity elsewhere in the $z$ plane having cuts along $(-\infty,0]$ and $[1,\infty)$; see \cite[Chap. 10, \S 10.1]{Olver:1997:ASF}. For $1 \leq z < \infty$ ($-\infty < \zeta \leq 0$) one can use 
%%%%%%%%%%%%%%%%
\begin{equation}
\label{eq52}
\frac{2}{3}(-\zeta)^{3/2}
=\int_{1}^{z}\frac{\left(t^2-1\right)^{1/2}}{t}dt
=\left(z^{2}-1\right)^{1/2}
-\mathrm{arcsec}(z).
\end{equation}
%%%%%%%%%%%%%%%%%%

For the LG coefficients it is convenient to work with the new variable 
%%%%%%%%%%%%%%%%%%
\begin{equation}  
\label{eq53}
\beta=\left(1-z^{2}\right)^{-1/2}.
\end{equation}
%%%%%%%%%%%%%%%%%%
Here the branch of the square root is positive for $-1<z<1$ and is continuous in the plane having a cuts along $(-\infty,-1]$ and $[1,\infty)$. Note that $\beta \rightarrow 0$ as $z \rightarrow \infty$. The coefficients in (\ref{eq11}) and (\ref{eq12}) are then given in \cite[Eqs. (61) - (63)]{Dunster:2022:LAS}, and from this we have $\hat{E}_{s}(z)=\mathrm{E}_{s}(\beta)$, which are polynomials in $\beta$ given by
%%%%%%%%%%%%%%%%%%%
\begin{equation}  
\label{eq55}
\mathrm{E}_{1}(\beta)=\tfrac{1}{24}\beta
\left(5\beta^{2}-3\right),
\quad
\mathrm{E}_{2}(\beta)=
\tfrac{1}{16}\beta^{2}\left(\beta^{2}-1\right)
\left(5\beta^{2}-1\right),
\end{equation}
%%%%%%%%%%%%%%%%%%%
and for $s=2,3,4,\ldots$
%%%%%%%%%%%%%%%%%%%
\begin{equation}  
\label{eq57}
\mathrm{E}_{s+1}(\beta) =
\frac{1}{2} \beta^{2} \left(\beta^{2}-1 \right)\mathrm{E}_{s}^{\prime}(\beta)
+\frac{1}{2}\int_{0}^{\beta}
p^{2}\left(p^{2}-1 \right)
\sum\limits_{j=1}^{s-1}
\mathrm{E}_{j}^{\prime}(p)
\mathrm{E}_{s-j}^{\prime}(p) dp.
\end{equation}
%%%%%%%%%%%%%%%%%%%
Note that $\lim_{z\to\infty}\hat{E}_{s}(z)=\mathrm{E}_{s}(0)=0$ for $s=1,2,3,\ldots$.

From \cite[Eq. (5.26)]{Dunster:2021:SEB} the connection constants $\lambda_{\pm 1}(u)=\lambda_{\pm 1}(\nu)$ appearing in (\ref{eq13}) are the same, and explicitly given by
%%%%%%%%%%%%%%%%%%%
\begin{equation}
\label{eq58}
\lambda_{\pm 1}(\nu)
=\left( \frac{1}{2\pi {\nu }}\right)^{1/2}
\frac{e^{{\nu }}\Gamma (\nu +1) }{{\nu }^{{\nu }}}.
\end{equation}
%%%%%%%%%%%%%%%%%%%

In the notation of \cref{sec2} we have $z^{(0)}=0$, $z^{(\pm 1)}= \mp i \infty$. The Bessel functions $z^{1/2}J_{\nu}(\nu z)$, $z^{1/2}H_{\nu}^{(1)}(\nu z)$ and $z^{1/2}H_{\nu }^{(2)}(\nu z)$ are multiples of $W_{0}(\nu,z)$, $W_{-1}(\nu,z)$ and $W_{1}(\nu,z)$, respectively, since each pair is the unique solution which is recessive at one of the singularities $z=0,\pm i \infty$. Specifically for the latter two, we have from \cite[Eqs. (5.16), (5.22) and (5.23)]{Dunster:2021:SEB}
% From 
% %%%%%%%%%%%%%%%%
% \begin{equation}
% \label{eq52a}
% \xi =\pm iz \mp \tfrac12 \pi i + \mathcal{O}\left(z^{-1}\right)
% \quad (z \to \pm i \infty)
% \end{equation}
% %%%%%%%%%%%%%%%%%%
% and
% %%%%%%%%%%%%%%%%%%%
% \begin{equation}
% \label{eq58a}
% z^{1/2}H_{\nu}^{(1,2)}(\nu z)
% \sim \sqrt{\frac{2}{\nu \pi}}
% \exp\left\{\pm i\nu z \mp \frac12\nu\pi i
% \mp \frac14\pi i\right\}
% \quad (z \to \pm i \infty)
% \end{equation}
% %%%%%%%%%%%%%%%%%%%
% we find by matching solutions recessive at this singularity
% USE f, W[+-1] defn, E(infty)=0, 
%%%%%%%%%%%%%%%%%%%
\begin{equation}
\label{eq58b}
z^{1/2}H_{\nu}^{(1,2)}(\nu z)
= \mp i \sqrt{\frac{2}{\nu \pi}}
W_{\mp 1}(\nu,z).
\end{equation}
%%%%%%%%%%%%%%%%%%%
Thus from (\ref{eq18}), (\ref{eq58}) and (\ref{eq58b}) we have the identifications
%%%%%%%%%%%%%%%%%%%
\begin{equation}
\label{eq59}
z^{1/2}H_{\nu}^{(1,2)}(\nu z)
=2 e^{\mp \pi i/3}
\mathcal{Y}(\nu,z)
\mathrm{Ai}_{\mp 1}\left( \nu^{2/3}
\mathcal{Z}(\nu,z)\right),
\end{equation}
%%%%%%%%%%%%%%%%%%%
% and
% %%%%%%%%%%%%%%%%%%%
% \begin{equation}
% \label{eq60}
% z^{1/2}H_{\nu}^{(2)}(\nu z)
% =2 e^{\pi i/3}
% \mathcal{Y}(\nu,z)
% \mathrm{Ai}_{1}\left( \nu^{2/3}
% \mathcal{Z}(\nu,z)\right)
% \end{equation}
% %%%%%%%%%%%%%%%%%%%
where we have chosen for convenience $c(\nu)=\sqrt{2 \pi \nu}\, \lambda_{\pm 1}(\nu)$. 

From (\ref{eq16}), (\ref{eq59}) and the Bessel connection formula 
%%%%%%%%%%%%%%%%%%%
\begin{equation}
\label{eq60a}
J_{\nu}(\nu z)
=\tfrac12\left\{H_{\nu }^{(1)}(\nu z)
+H_{\nu }^{(2)}(\nu z)\right\},
\end{equation}
%%%%%%%%%%%%%%%%%%%
we obtain the identity
%%%%%%%%%%%%%%%%%%%
\begin{equation}
\label{eq61}
z^{1/2}J_{\nu}(\nu z)
=\mathcal{Y}(\nu,z) \mathrm{Ai}
\left( \nu^{2/3}\mathcal{Z}(\nu,z)\right).
\end{equation}
%%%%%%%%%%%%%%%%%%%
% Therefore as required the Bessel function recessive at $z^{(0)}=0$ ($\zeta=+\infty$) is a multiple of the Airy function expansion recessive at the same singularity (see (\ref{eq19})).
All other Bessel functions can be expressed similarly. For example from (\ref{eq59}) and \cite[Eqs. 9.2.10 and 10.4.4]{NIST:DLMF} we have
%%%%%%%%%%%%%%%%%%%
\begin{equation}
\label{eq62}
z^{1/2}Y_{\nu}(\nu z)
=-\mathcal{Y}(\nu,z) \mathrm{Bi}
\left( \nu^{2/3}\mathcal{Z}(\nu,z)\right).
\end{equation}
%%%%%%%%%%%%%%%%%%%

Now consider determining the coefficients in the expansion (\ref{eq32a}). For these it is helpful to use the variable given by
%%%%%%%%%%%%%%%%%%%
\begin{equation}
\label{eq65}
\sigma(z)=\zeta^{1/2}\beta
=\left(\frac{\zeta}{ 1-z^{2}}\right)^{1/2},
\end{equation}
%%%%%%%%%%%%%%%%%%%
which is analytic at $z=1$, and indeed from (\ref{eq51})
%%%%%%%%%%%%%%%%%%%
\begin{equation}
\label{eq66}
\sigma(z)
=\frac{1}{2^{1/3}}
-\frac{{2}^{2/3}}{5}(z-1)
+\mathcal{O}\left\{(z-1)^2\right\}.
\end{equation}
%%%%%%%%%%%%%%%%%%%

Now from (\ref{eq53}) - (\ref{eq57}) we observe that each $\zeta^{1/2}\mathrm{E}_{2s+1}(\beta)$ ($s=0,1,2,\ldots$) is a polynomial in $\sigma$ and a rational function of $\zeta$, and hence so is $\xi \hat{E}_{2s+1}(z)=\frac{2}{3}\zeta^{3/2}\mathrm{E}_{2s+1}(\beta)$. Then from \cref{thm2.2} we find that $\Upsilon_{s}(z)$ are polynomials in $\sigma$ and rational functions of $\zeta$, with the first four given by
%%%%%%%%%%%%%%%%%%%
\begin{equation}
\label{eq67}
\Upsilon_{1}(z)
=\frac{1}{ 48\,\zeta^{2}}
\left\{
10\,\sigma^{3}-6\,\sigma\,\zeta-5
\right\},
\end{equation}
%%%%%%%%%%%%%%%%%%%
%%%%%%%%%%%%%%%%%%%
\begin{multline}
\label{eq68}
\Upsilon_{2}(z)
=\frac{1}{ 11520\,\zeta^{5}}
\left\{
11050\,\sigma^{9}-19890\,\sigma^{7}\zeta
+9558\,\sigma^{5}\zeta^{2}
-125\,\sigma^{6}+150\,\sigma^{4}\zeta
\right.
\\
\left.
-45\,\sigma^{2}\zeta^{2}-250 \left( 3\,\zeta^{3}-2 \right) {
\sigma}^{3}-300\,\sigma\,\zeta-1600
\right\},
\end{multline}
%%%%%%%%%%%%%%%%%%%
%%%%%%%%%%%%%%%%%%%
\begin{multline}
\label{eq69}
\Upsilon_{3}(z)
=\frac{1}{5806080 \,\zeta^{8}}
\left\{
156539250\,\sigma^{15}-469617750\,\sigma^{13}\zeta+509154660\,{
\sigma}^{11}\zeta^{2}
\right.
\\
-580125\,\sigma^{12}
+1392300\,\sigma^{10}\zeta-1128330
\,\sigma^{8}\zeta^{2}
-140\left( 1681389\,\zeta^{3}-8350 \right) 
\sigma^{9}
\\
+90\left( 450441\,\zeta^{3}-23380 \right) \sigma^{7}\zeta-378\left( 3219\,\zeta^{3}
-2680 \right) \sigma^{5}\zeta^{2}
\\
+147
 \left( 2316\,\zeta^{3}-625 \right) \sigma^{6}
-7875\left( 3\,\zeta^{3}-14 \right)
\sigma^{4}\zeta
-33075\sigma^{2}\zeta^{2}
\\
\left.
-6720 \left( 12\,\zeta^{3}-125 \right) 
\sigma^{3}
-504000\,\sigma\,\zeta-5398750
\right\},
\end{multline}
%%%%%%%%%%%%%%%%%%%
and
%%%%%%%%%%%%%%%%%%%
\begin{multline}
\label{eq70}
\Upsilon_{4}(z)
=\frac{1}{2786918400 \,\zeta^{11}}
\left\{
5192227676250 \, \sigma^{21}
- 21807356240250 \,\sigma^{19} \zeta 
\right.
\\
+ 36744870767850 \, \sigma^{17} \zeta^{2}
-8468000625 \,\sigma^{18}
+30484802250 \,\sigma^{16}\zeta
-42732195975 \,\sigma^{14}\zeta^{2}
\\
-210 \left( 150024588453\,
\zeta^{3}-74818750 \right) \sigma^{15}
+3150
\left( 4547729961\,\zeta^{3}-14963750 \right) \sigma^{13}
\zeta
\\
-9450\left( 342342531\,\zeta^{3}-5408660 \right) \sigma^{11}\zeta^{2}
+35 \left( 832214916\,\zeta^{3}-11624375 \right) \sigma^{12}
\\
-3 \left( 3241689237 \,
\zeta^{3}-325482500 \right) \sigma^{10} \zeta
+1350 \left( 1001703\,\zeta^{3}-586285 \right) \sigma^{8}\zeta^{2}
\\
+350 \left( 827921169\,\zeta^{6}
-67546260\,\zeta^{3}+5347750 \right) \sigma^{9}
\\
-90 \left( 50723955\,\zeta^{6}
-45297318\,\zeta^{3}+37434250\right)
\sigma^{7}\zeta
-18900 \left( 6513\,\zeta^{3}-85835
 \right) \sigma^{5}\zeta^{2}
 \\
 -105 \left( 375777\,\zeta^{6}-2275980\,\zeta^{3}
 +2600000 \right) \sigma^{6}
 -1575\left( 10563\,\zeta^{3}-208000 \right) 
\sigma^{4}\zeta
\\
\left.
-98280000 \, \sigma^{2}\zeta^{2}
-17500 \left( 7389\,\zeta^{3}-
246800 \right)
\sigma^{3}-2591400000\,\sigma\,\zeta
-43222750000
\right\}.
\end{multline}
%%%%%%%%%%%%%%%%%%%

Also from \cref{thm2.2} we have that these are analytic at $z=1$ ($\zeta=0$), and to confirm this we find from (\ref{eq51}), (\ref{eq65}), and (\ref{eq67}) - (\ref{eq70}) as $z \to 1$
%%%%%%%%%%%%%%%%%%%
\begin{equation}
\label{eq71}
\Upsilon_{1}(z)=
\tfrac{1}{70} 2^{1/3}
+\mathcal{O}(z-1),
\end{equation}
%%%%%%%%%%%%%%%%%%%
%%%%%%%%%%%%%%%%%%%
\begin{equation}
\label{eq72}
\Upsilon_{2}(z)=
-\tfrac{82}{73125} 2^{1/3}
+\mathcal{O}(z-1),
\end{equation}
%%%%%%%%%%%%%%%%%%%
%%%%%%%%%%%%%%%%%%%
\begin{equation}
\label{eq73}
\Upsilon_{3}(z)=
\tfrac{53780996}{127020403125} 2^{1/3}
+\mathcal{O}(z-1),
\end{equation}
%%%%%%%%%%%%%%%%%%%
and
%%%%%%%%%%%%%%%%%%%
\begin{equation}
\label{eq74}
\Upsilon_{4}(z)=
-\tfrac{52568866144}{142468185234375} 2^{1/3}
+\mathcal{O}(z-1).
\end{equation}
%%%%%%%%%%%%%%%%%%%

Furthermore as $z \to \infty$ we have
%%%%%%%%%%%%%%%%%%%
\begin{equation}
\label{eq75a}
\Upsilon_{1}(z)=
\tfrac{1}{108}12^{2/3}z^{-4/3}
\left\{1+\mathcal{O}\left(z^{-1}\right)\right\},
\end{equation}
%%%%%%%%%%%%%%%%%%%
%%%%%%%%%%%%%%%%%%%
\begin{equation}
\label{eq75b}
\Upsilon_{2}(z)=
-\tfrac{4}{729}12^{2/3}z^{-10/3}
\left\{1+\mathcal{O}\left(z^{-1}\right)\right\},
\end{equation}
%%%%%%%%%%%%%%%%%%%
%%%%%%%%%%%%%%%%%%%
\begin{equation}
\label{eq75c}
\Upsilon_{3}(z)=
\tfrac{5218}{295245}12^{2/3}z^{-16/3}
\left\{1+\mathcal{O}\left(z^{-1}\right)\right\},
\end{equation}
%%%%%%%%%%%%%%%%%%%
and
%%%%%%%%%%%%%%%%%%%
\begin{equation}
\label{eq75d}
\Upsilon_{4}(z)=
-\tfrac{7672012}
{55801305}12^{2/3}z^{-22/3}
\left\{1+\mathcal{O}\left(z^{-1}\right)\right\},
\end{equation}
%%%%%%%%%%%%%%%%%%%
which in conjunction with (\ref{eq71}) - (\ref{eq74}) underlines the uniform validity of the expansion (\ref{eq32a}) at the turning point and for unbounded $z$.

Now for $J_{\nu}(\nu z)$ denote the $m$th positive zero by $z_{m}$, which depends of course on $\nu$ but we suppress this dependence for ease of notation. It follows that
%%%%%%%%%%%%%%%%%%%
\begin{equation}
\label{eq75e}
J_{\nu}\left(\nu z_{m}\right)=0,
\end{equation}
%%%%%%%%%%%%%%%%%%%
and we let the corresponding value of $\zeta$ be given by $\zeta=\zeta_{m}$. From (\ref{eq61}) we must solve
%%%%%%%%%%%%%%%%%%%
\begin{equation}
\label{eq82a}
\mathcal{Z}(\nu,z_{m})=\nu^{-2/3}\mathrm{a}_{m},
\end{equation}
%%%%%%%%%%%%%%%%%%%
where we recall that $\mathrm{a}_{m}$ denotes the $m$th (negative) zero of the Airy function $\mathrm{Ai}(x)$, ordered by increasing absolute values. We seek the coefficients in the asymptotic expansion (\ref{eq03}), where $j_{\nu,m}=\nu z_{m}$.

On referring to (\ref{eq32a}) and (\ref{eq82a}) we deduce that
%%%%%%%%%%%%%%%%%%%
\begin{equation}
\label{eq84}
\nu^{-2/3}\mathrm{a}_{m} 
\sim \zeta_{m}
+\sum_{s=1}^{\infty}
\frac{\Upsilon_{s}(z_{m})}{\nu^{2s}},
\end{equation}
%%%%%%%%%%%%%%%%%%%
recalling that $\zeta_{m}=\zeta(z_{m})$. By neglecting all but the leading terms in (\ref{eq03}) and (\ref{eq84}) for large $\nu$ leads to the first order approximation $\zeta_{m} \approx  \zeta_{m,0}$, where $\zeta_{m,0}=\nu^{-2/3}\mathrm{a}_{m}$, as defined by (\ref{eq04}). Hence $z_{m,0}$ satisfies the equation
%%%%%%%%%%%%%%%%%%%
\begin{equation}
\label{eq76}
\zeta\left(z_{m,0}\right)
=\zeta_{m,0}
=\nu^{-2/3}\mathrm{a}_{m}.
\end{equation}
%%%%%%%%%%%%%%%%%%%
From this and (\ref{eq52}), we perceive that $z_{m,0}$ is the unique solution lying in $(1,\infty)$ of the implicit equation
%%%%%%%%%%%%%%%%%%%
\begin{multline}
\label{eq77}
\int_{1}^{z_{m,0}}
\frac{\left(t^{2}-1\right)^{1/2}}{t}\,dt
=\left(z_{m,0}^{2}-1\right)^{1/2}
-\mathrm{arcsec}\left( z_{m,0} \right)
\\
= \frac{2}{3}\left(-\zeta_{m,0}\right)^{3/2}
=\frac{2}{3 \nu}\left|\mathrm{a}_{m}\right|^{3/2}.
\end{multline}
%%%%%%%%%%%%%%%%%%%
Note that $z_{m,0}$ increases monotonically as $\zeta_{m,0}$ decreases monotonically,
%%%%%%%%%%%%%%%%%%%
\begin{equation}
\label{eq78}
z_{m,0} = 1 -2^{-1/3}\zeta_{m,0}
+\mathcal{O}\left(\zeta_{m,0}^{2}\right)
\quad (\zeta_{m,0} \to 0),
\end{equation}
%%%%%%%%%%%%%%%%%%%
%%%%%%%%%%%%%%%%%%%
\begin{equation}
\label{eq79}
z_{m,0} = \tfrac{2}{3}\left|\zeta_{m,0}\right|^{3/2}
+\tfrac{1}{2}\pi+\mathcal{O}\left(
\left|\zeta_{m,0}\right|^{-3/2}
\right)
\quad (\zeta_{m,0} \to -\infty),
\end{equation}
%%%%%%%%%%%%%%%%%%%
and $1<z_{m,0}<\infty$ since $-\infty<\zeta_{m,0}<0$.

In terms of solving (\ref{eq77}) numerically observe that for any value of $z_{m,0} \in (1,\infty)$
%%%%%%%%%%%%%%%%%%%
\begin{equation}
\label{eq80}
z_{m,0} -1- \tfrac{1}{2}\pi
< \left(z_{m,0}^{2}-1\right)^{1/2}
-\mathrm{arcsec}\left( z_{m,0} \right)
< z_{m,0},
\end{equation}
%%%%%%%%%%%%%%%%%%%
and therefore we can solve (\ref{eq77}) in the interval
%%%%%%%%%%%%%%%%%%%
\begin{equation}
\label{eq81}
\max\left\{
\tfrac{2}{3}\left|\zeta_{m,0}\right|^{3/2},1\right\}
<z_{m,0}< \tfrac{2}{3}\left|\zeta_{m,0}\right|^{3/2}
+1 + \tfrac{1}{2}\pi.
\end{equation}
%%%%%%%%%%%%%%%%%%%

Now dividing (\ref{eq61}) by (\ref{eq62}) and using the Airy and Bessel function Wronskian relations \cite[Eqs. 9.2.7 and 10.5.2]{NIST:DLMF} we have
%%%%%%%%%%%%%%%%%%%
\begin{equation}
\label{eq83}
\frac{d}{dz}\mathcal{Z}(\nu,z) 
= \frac{2 \nu^{1/3}
\mathrm{Bi}^{2}\left( \nu^{2/3}\mathcal{Z}(\nu,z)\right)}
{z Y_{\nu}^{2}(\nu z)}.
\end{equation}
%%%%%%%%%%%%%%%%%%%
At the zero $z=z_{m}$ of $J_{\nu}(\nu z)$ we know that $Y_{\nu}(\nu z_{m}) \neq 0$, and also $\mathrm{Bi}(\nu^{2/3}\mathcal{Z}(\nu,z_{m}))=\mathrm{Bi}(\mathrm{a}_{m}) \neq 0$, from well-known properties of both sets of functions, and consequently from (\ref{eq83}) the derivative of $\mathcal{Z}(\nu,z)$ does not vanish at this point. Therefore from the inverse function theorem for analytic functions \cite[\S 6.22]{Copson:1935:ITF} $z_{m}$ is the unique solution of (\ref{eq82a}), and is analytic as a function of $\zeta_{m,0}$ at $\zeta_{m,0}=0$ ($z_{m,0}=1$). Hence by \cite[Thm. 3.1]{Dunster:2021:NKF} the coefficients $z_{m,s}$ in (\ref{eq03}), which we know are meromorphic at $z=1$, actually all have a removable singularity at this point. We shall determine these coefficients explicitly next.

Let $\delta = \nu^{-2}$ and define for arbitrary positive integer $s$
%%%%%%%%%%%%%%%%%%%
\begin{equation}
\label{eq86}
\breve{\zeta}(\delta)=
\zeta(z_{m,0}+z_{m,1}\delta + z_{m,2}\delta^{2}
+ \cdots +z_{m,s}\delta^{s}),
\end{equation}
%%%%%%%%%%%%%%%%%%%
and
%%%%%%%%%%%%%%%%%%%
\begin{equation}
\label{eq87}
\breve{\Upsilon}_{l}(\delta)=
\Upsilon_{l}(z_{m,0}+z_{m,1}\delta + z_{m,2}\delta^{2}
+ \cdots +z_{m,s}\delta^{s})
\quad (l=1,2,3,\ldots,s).
\end{equation}
%%%%%%%%%%%%%%%%%%%
Note that the series $\sum_{s=0}^{\infty}z_{m,s}\delta^{s}$ is divergent for $\delta \neq 0$, but this is not an issue in (\ref{eq86}) and (\ref{eq87}) since we truncate this after a finite number of terms. 

Now in (\ref{eq84}) replace $z_{m}$ by its truncated asymptotic expansion (\ref{eq03}) (recalling that $j_{\nu,m}=\nu z_{m}$), let $\nu^{-2}=\delta$, expand each term as a Maclaurin series in $\delta$ using (\ref{eq86}) and (\ref{eq87}), and for $s=1,2,3,\ldots$ set the coefficient of $\delta^{s}$ to zero. This yields
%%%%%%%%%%%%%%%%%%%
\begin{equation}
\label{eq85}
\frac{1}{s!}\breve{\zeta}^{(s)}(0)
+\frac{1}{(s-1)!}\breve{\Upsilon}_{1}^{(s-1)}(0)
+\frac{1}{(s-2)!}\breve{\Upsilon}_{2}^{(s-2)}(0)
+ \cdots
+\frac{1}{0!}\breve{\Upsilon}_{s}(0)
=0.
\end{equation}
%%%%%%%%%%%%%%%%%%%
From (\ref{eq86}) and Faà Di Bruno’s Formula \cite[Eq. 1.4.13]{NIST:DLMF}
%%%%%%%%%%%%%%%%%%%
\begin{equation}
\label{eq88}
\frac{1}{s!}\breve{\zeta}^{(s)}(0)
=\mathlarger{\sum}
\frac{\zeta^{(k)}(z_{m,0})}
{q_{1}!q_{2}!\cdots q_{s}!}
\left\{z_{m,1}\right\}^{q_{1}}
\left\{z_{m,2}\right\}^{q_{2}}
\cdots\left\{z_{m,s}\right\}^{q_{s}},
\end{equation}
%%%%%%%%%%%%%%%%%%%
and for $j=1,2,3,\ldots,s-1$ we likewise have from (\ref{eq87})
%%%%%%%%%%%%%%%%%%%
\begin{equation}
\label{eq89}
\frac{1}{j!}\breve{\Upsilon}_{s-j}^{(j)}(0)
=\mathlarger{\sum}
\frac{\Upsilon_{s-j}^{(k)}(z_{m,0})}
{q_{1}!q_{2}!\cdots q_{j}!}
\left\{z_{m,1}\right\}^{q_{1}}
\left\{z_{m,2}\right\}^{q_{2}}
\cdots\left\{z_{m,j}\right\}^{q_{j}},
\end{equation}
%%%%%%%%%%%%%%%%%%%
where in (\ref{eq88}) the sum is over all nonnegative integers $q_{1}, q_{2}, \ldots , q_{s}$ that satisfy $q_{1}+2q_{2} + \cdots + s\, q_{s} =s$ and $k=q_{1}+q_{2} + \cdots + q_{s}$, and similarly for (\ref{eq89}) with $s$ replaced by $j$.

Now in (\ref{eq88}) for each fixed $s=1,2,3,\ldots$ the only occurrence of $z_{m,s}$ is when $q_{s}=1$ and all the other $q_{l}$ ($l=1,2,\ldots,s-1$) are necessarily zero, and hence in this case $k=1$. The remaining terms in the sum have $q_{s}=0$, and so on pulling out the solitary term having $q_{s}=1$ yields
%%%%%%%%%%%%%%%%%%%
\begin{multline}
\label{eq90}
\frac{1}{s!}\breve{\zeta}^{(s)}(0)
=z_{m,s} \zeta^{\prime}(z_{m,0})
\\
+\mathlarger{\sum}
\frac{\zeta^{(k)}(z_{m,0})}
{q_{1}!q_{2}!\cdots q_{s-1}!}
\left\{z_{m,1}\right\}^{q_{1}}
\left\{z_{m,2}\right\}^{q_{2}}
\cdots\left\{z_{m,s-1}\right\}^{q_{s-1}},
\end{multline}
%%%%%%%%%%%%%%%%%%%
where now the sum is over all nonnegative integers $q_{1}, q_{2}, \ldots , q_{s-1}$ that satisfy $q_{1}+2q_{2} + \cdots + (s-1)\, q_{s-1} =s$ and $k=q_{1}+q_{2} + \cdots + q_{s-1}$.

We deduce that $z_{m,s}$ only appears as a solitary term in the first term on the LHS of (\ref{eq85}), and consequently we can solve for it explicitly, yielding for $m=1,2,3,\ldots$, $s=1,2,3,\ldots$
%%%%%%%%%%%%%%%%%%%
\begin{multline}
\label{eq91}
z_{m,s} 
= -\frac{1}{\zeta^{\prime}(z_{m,0})}
\Biggl\{\mathlarger{\sum}
\frac{\zeta^{(k)}(z_{m,0})}
{q_{1}!q_{2}!\cdots q_{s-1}!}
\left\{z_{m,1}\right\}^{q_{1}}
\left\{z_{m,2}\right\}^{q_{2}}
\cdots\left\{z_{m,s-1}\right\}^{q_{s-1}}
\Biggr.
\\
\Biggl.
+\mathlarger{\sum}_{j=1}^{s-1}
\mathlarger{\sum}
\frac{\Upsilon_{s-j}^{(k)}(z_{m,0})}
{q_{1}!q_{2}!\cdots q_{j}!}
\left\{z_{m,1}\right\}^{q_{1}}
\left\{z_{m,2}\right\}^{q_{2}}
\cdots\left\{z_{m,j}\right\}^{q_{j}}
+\Upsilon_{s}(z_{m,0})
\Biggr\},
\end{multline}
%%%%%%%%%%%%%%%%%%%
with all sums being null for $s=1$, in this case leaving inside the braces the single term $\Upsilon_{s}(z_{m,0})=\breve{\Upsilon}_{s}(0)$.

Recalling the notation $j_{\nu,m}=\nu z_{m}$, we have proven the following main result of this paper.
\begin{theorem}
\label{thm3.1}
Let $0>\mathrm{a}_{1}>\mathrm{a}_{2}>\ldots$ be the zeros of $\mathrm{Ai}(x)$, and $\zeta$ be given by (\ref{eq51}). For $m=1,2,3,\dots$ define $\zeta_{m,0}=\nu^{-2/3}\mathrm{a}_{m}$ and let $z_{m,0} \in (1,\infty)$ be the unique solution to (\ref{eq77}). Let $z_{m,s}$ ($m=1,2,3,\dots$, $s=1,2,3,\dots$) be given recursively by (\ref{eq91}), with the first sum, and the inner sum of the second double sum, taken as described after (\ref{eq89}) and (\ref{eq90}). In (\ref{eq91}) $\Upsilon_{s}(z)$ are defined as in \cref{thm2.2}, with coefficients $\hat{E}_{s}(z)=\mathrm{E}_{s}(\beta)$ given by (\ref{eq53}) - (\ref{eq57}). Then the $m$th positive zero $j_{\nu,m}$ of $J_{\nu}(z)$ possesses the expansion
%%%%%%%%%%%%%%%%%%%%%%%%%
\begin{equation}
\label{eq93a}
j_{\nu,m} \sim
\nu \sum_{s=0}^{\infty}
\frac{z_{m,s}}{\nu^{2s}}
\quad (\nu \to \infty),
\end{equation}
%%%%%%%%%%%%%%%%%%%%%%%%%
uniformly for $m=1,2,3,\dots$.
\end{theorem}

Now, in addition to our previous notation $\zeta_{m,0}=\zeta(z_{m,0})$, define
%%%%%%%%%%%%%%%%%%%
\begin{equation}
\label{eq92}
\zeta_{m,0}^{\prime}
=\zeta'\left(z_{m,0}\right),  \;
\zeta_{m,0}^{\prime \prime}
=\zeta''\left(z_{m,0}\right),  \;
\ldots,
\end{equation}
%%%%%%%%%%%%%%%%%%%
and similarly for $s=1,2,3,\ldots$ let
%%%%%%%%%%%%%%%%%%%
\begin{equation}
\label{eq93}
\Upsilon_{m,s}
=\Upsilon_{s}\left(z_{m,0}\right), \;
\Upsilon_{m,s}^{\prime}
=\Upsilon_{s}'\left(z_{m,0}\right),  \;
\Upsilon_{m,s}^{\prime \prime}
=\Upsilon_{s}''\left(z_{m,0}\right),  \;
\ldots.
\end{equation}
%%%%%%%%%%%%%%%%%%%
We then get from (\ref{eq91})
%%%%%%%%%%%%%%%%%%%
\begin{equation}
\label{eq94}
z_{m,1}=
-\frac{\Upsilon_{m,1}}
{\zeta'_{m,0}},
\end{equation}
%%%%%%%%%%%%%%%%%%%
%%%%%%%%%%%%%%%%%%%
\begin{equation}
\label{eq95}
z_{m,2} =-\frac{1}
{2 \zeta'_{m,0}}
\left\{
z_{m,1}^{2} \, 
\zeta''_{m,0}+2 z_{m,1}\, \Upsilon'_{m,1}
+2\Upsilon_{m,2}
\right\},
\end{equation}
%%%%%%%%%%%%%%%%%%%
%%%%%%%%%%%%%%%%%%%
\begin{multline}
\label{eq96}
z_{m,3} =-\frac{1}
{6 \zeta'_{m,0}}
\left\{
z_{m,1}^3 \,\zeta'''_{m,0} + 6z_{m,1}\,z_{m,2}\,
\zeta''_{m,0} + 3z_{m,1}^2\Upsilon''_{m,1}
\right.
\\
\left. 
+ 6z_{m,2}\Upsilon'_{m,1} 
+ 6z_{m,1}\Upsilon'_{m,2}  
+ 6\Upsilon_{m,3} 
\right\},
\end{multline}
%%%%%%%%%%%%%%%%%%%
and
%%%%%%%%%%%%%%%%%%%
\begin{multline}
\label{eq97}
z_{m,4} =-\frac{1}
{24 \zeta'_{m,0}}
\left\{
z_{m,1}^4 \,\zeta^{(4)}_{m,0}
+12 z_{m,1}^{2}\, z_{m,2}\,\zeta'''_{m,0}
+24z_{m,1}\, z_{m,3}\,\zeta''_{m,0}
+12z_{m,2}^2\,\zeta''_{m,0}
\right.
\\
+4z_{m,1}^3 \,\Upsilon'''_{m,1}
+24z_{m,1}\,z_{m,2}\,\Upsilon''_{m,1}
+12z_{m,1}^2 \,\Upsilon''_{m,2}
+24z_{m,3}\,\Upsilon_{m,1}
\\
\bigl. 
+24z_{m,2}\,\Upsilon_{m,2}
+24z_{m,1}\,\Upsilon_{m,3}
+24\Upsilon_{m,4}
\Bigr\}.
\end{multline}
%%%%%%%%%%%%%%%%%%%

Consider now computation of the derivatives appearing in the expressions (\ref{eq94}) - (\ref{eq97}). First using (\ref{eq51}) and (\ref{eq65})
%%%%%%%%%%%%%%%%%%%
\begin{equation}
\label{eq98}
\zeta'(z)=
-\{z\sigma(z)\}^{-1},
\end{equation}
%%%%%%%%%%%%%%%%%%%
and hence from this and (\ref{eq65})
%%%%%%%%%%%%%%%%%%%
\begin{equation}
\label{eq99}
\sigma'(z)=\frac{2\,{z}^{2} \sigma^{3}(z)-1}
{ 2z \zeta(z)}.
\end{equation}
%%%%%%%%%%%%%%%%%%%
We earlier showed that each $\Upsilon_{s}(z)$ is a polynomial in $\sigma$ and rational function of $\zeta$. Thus from (\ref{eq98}) and (\ref{eq99}) all derivatives of $\zeta(z)$, $\sigma(z)$ and $\Upsilon_{s}(z)$ at $z=z_{m,0}$ can, via the chain rule, be expressed explicitly as rational functions of $z_{m,0}$, $\zeta_{m,0}$ and $\sigma_{m,0}$, where $\sigma_{m,0}=\sigma(z_{m,0})$. From (\ref{eq91}) it follows the same is true of the coefficients $z_{m,s}$ that we are evaluating.

Recall that $\zeta_{m,0}$ is given by (\ref{eq76}), $z_{m,0}$ is computed from (\ref{eq77}), and from (\ref{eq65}) we then have
%%%%%%%%%%%%%%%%%%%
\begin{equation}
\label{eq100}
\sigma_{m,0}
=\frac{1}{\nu^{1/3}}
\left(\frac{|\mathrm{a}_{m}|}
{z_{m,0}^{2}-1}\right)^{1/2}.
\end{equation}
%%%%%%%%%%%%%%%%%%%
It follows from (\ref{eq98}) and (\ref{eq99}) that
%%%%%%%%%%%%%%%%%%%
\begin{equation}
\label{eq101}
\zeta_{m,0}^{\prime\prime}
=\zeta''\left(z_{m,0}\right)
=\frac{2 z_{m,0}^{2} \, \sigma_{m,0}^{3}
+2 \sigma_{m,0} \, \zeta_{m,0}-1}
{2z_{m,0}^{2}\,\sigma_{m,0}^{2}\,\zeta_{m,0}},
\end{equation}
%%%%%%%%%%%%%%%%%%%
and from (\ref{eq67}), (\ref{eq98}) and (\ref{eq99})
%%%%%%%%%%%%%%%%%%%
\begin{equation}
\label{eq102}
\Upsilon_{m,1}^{\prime}
=\Upsilon_{1}'\left(z_{m,0}\right)
=\frac {30z_{m,0}^{2} \, \sigma_{m,0}^{6}
-6z_{m,0}^{2} \,\sigma_{m,0}^{4} \,\zeta_{m,0}
+5\sigma_{m,0}^{3}
-3\sigma_{m,0} \, \zeta_{m,0}-10}
{48z_{m,0} \, \sigma_{m,0} \,\zeta_{m,0}^{3}}.
\end{equation}
%%%%%%%%%%%%%%%%%%%

Hence from (\ref{eq94}) and (\ref{eq95}) the first two coefficients in our main expansion (\ref{eq93a}) are given explicitly by
%%%%%%%%%%%%%%%%%%%
\begin{equation}
\label{eq103}
z_{m,1}=
\frac{z_{m,0} \,\sigma_{m,0}}{48  
\zeta_{m,0}^{2}}
\left\{ 10\sigma_{m,0}^3
- 6\sigma_{m,0} \,\zeta_{m,0} - 5 \right\},
\end{equation}
%%%%%%%%%%%%%%%%%%%
and 
%%%%%%%%%%%%%%%%%%%
\begin{multline}
\label{eq105}
z_{m,2}=
\frac{z_{m,0} \,\sigma_{m,0}}{46080
\zeta_{m,0}^{5}}
\left\{
200\sigma_{m,0}^{9}\left( 35z_{m,0}^{2}
+221 \right) 
-80\sigma_{m,0}^{7} \, \zeta_{m,0}\left( 75z_{m,0}^{2}+982\right) 
\right.
\\
-4000\sigma_{m,0}^{6} \, z_{m,0}^{2}
+24\sigma_{m,0}^{5} \, \zeta_{m,0}^{2}\left( 45 z_{m,0}^{2}+1543\right)
+200\sigma_{m,0}^{4} \, \zeta_{m,0}
\left(6 z_{m,0}^{2}-5 \right) 
 \\
 \left.
+600\sigma_{m,0}^{2} \, \zeta_{m,0}^{2}
+10 \sigma_{m,0}^{3}\left(25 z_{m,0}^{2} 
-264 \zeta_{m,0}^{3}\right) 
+250  \sigma_{m,0} \,\zeta_{m,0}-5525
\right\}.
\end{multline}
%%%%%%%%%%%%%%%%%%%
All the other coefficients can be found in the same way, with the next two being given in the first Maple worksheet of \cref{secA}.

Note as $z_{m,0} \to 1$ ($\zeta_{m,0} \to 0$) which occurs when $\nu^{-2/3}\mathrm{a}_{m} \to 0$, for example $\nu \to \infty$ with $m$ is fixed, we have from (\ref{eq65}), (\ref{eq67}) - (\ref{eq70}), (\ref{eq77}) and (\ref{eq94}) - (\ref{eq99})
%%%%%%%%%%%%%%%%%%%
\begin{equation}
\label{eq106}
z_{m,1}=
\tfrac{1}{70} 
+\mathcal{O}\left(z_{m,0}-1\right),
\end{equation}
%%%%%%%%%%%%%%%%%%%
%%%%%%%%%%%%%%%%%%%
\begin{equation}
\label{eq107}
z_{m,2}=
-\tfrac{3781}{3185} \times 10^{-3}
+\mathcal{O}\left(z_{m,0}-1\right),
\end{equation}
%%%%%%%%%%%%%%%%%%%
%%%%%%%%%%%%%%%%%%%
\begin{equation}
\label{eq108}
z_{m,3}=
\tfrac{722735647}{163087925} 
\times 10^{-4}
+\mathcal{O}\left(z_{m,0}-1\right),
\end{equation}
%%%%%%%%%%%%%%%%%%%
and
%%%%%%%%%%%%%%%%%%%
\begin{equation}
\label{eq109}
z_{m,4}=
-\tfrac{56446083463751}{14841001175} 
\times 10^{-7}
+\mathcal{O}\left(z_{m,0}-1\right).
\end{equation}
%%%%%%%%%%%%%%%%%%%

We remark that there are cancellations of singularities at $z_{m,0}=1$, more so for the later coefficients. However we find numerically this is not an issue, since $z_{m,0}$ is never exactly equal to 1, and moreover it is close to that value only when $\nu$ is very large relative to $m$, and hence less accuracy is required for the later terms in the series (\ref{eq93a}).

Next, as $z_{m,0} \to \infty$ ($\zeta_{m,0} \to -\infty$), which occurs when $\nu^{-2/3}|\mathrm{a}_{m}| \to \infty$, for example $m \to \infty$ with $\nu$ fixed, we have
%%%%%%%%%%%%%%%%%%%
\begin{equation}
\label{eq110}
z_{m,1}=
\frac{1}{18 \, z_{m,0}} 
+\mathcal{O}\left(\frac{1}{z_{m,0}^{2}}\right),
\end{equation}
%%%%%%%%%%%%%%%%%%%
%%%%%%%%%%%%%%%%%%%
\begin{equation}
\label{eq111}
z_{m,2}=
-\frac{71}{1944 \, z_{m,0}^{3}} 
+\mathcal{O}\left(\frac{1}{z_{m,0}^{4}}\right),
\end{equation}
%%%%%%%%%%%%%%%%%%%
%%%%%%%%%%%%%%%%%%%
\begin{equation}
\label{eq112}
z_{m,3}=
\frac{6673}{58320 \, z_{m,0}^{5}} 
+\mathcal{O}\left(\frac{1}{z_{m,0}^{6}}\right),
\end{equation}
%%%%%%%%%%%%%%%%%%%
and
%%%%%%%%%%%%%%%%%%%
\begin{equation}
\label{eq113}
z_{m,4}=
-\frac{25500599}{29393280 \, z_{m,0}^{7}} 
+\mathcal{O}\left(\frac{1}{z_{m,0}^{8}}\right).
\end{equation}
%%%%%%%%%%%%%%%%%%%
Thus (\ref{eq106}) - (\ref{eq113}) illustrate the uniform nature of our approximations for $1<z_{m,0}<\infty$.

We mention that we could in place of (\ref{eq91}) use the Taylor series 
%%%%%%%%%%%%%%%%%%%
\begin{equation}
\label{eq114}
h\left(z_{m}\right) \sim 
h_{m,0}
+h_{m,0}^{\prime}
\sum_{s=1}^{\infty}
\frac{z_{m,s}}{\nu^{2s}}
+\frac{1}{2!}h_{m,0}^{\prime \prime}
\left\{\sum_{s=1}^{\infty}
\frac{z_{m,s}}{\nu^{2s}}\right\}^2
+ \ldots,
\end{equation}
%%%%%%%%%%%%%%%%%%%
in (\ref{eq84}) with $h=\zeta$ and $h=\Upsilon_{s}$. Then expand with the aid of \cite[Eqs. 1.9.60 - 1.9.62]{NIST:DLMF}, and equate like powers of $\nu^{-2s}$, set each coefficient of this series to zero, and arrive at (\ref{eq94}) - (\ref{eq97}). This method is useful to verify the first few terms given by (\ref{eq91}), but entails extensive algebraic manipulations, and does not give a recursion for the coefficients.

\subsection{Zeros of other Bessel functions}
\label{sec3.1}

From (\ref{eq62}) we find similarly to $J_{\nu}(\nu z)$ that uniform asymptotic expansions for the real zeros of $Y_{\nu}(\nu z)$ are the same as those given in \cref{thm3.1}, except in place of (\ref{eq76}) we take $\zeta_{m,0}=\nu^{-2/3}\mathrm{b}_{m}$, where $0>\mathrm{b}_{1}>\mathrm{b}_{2}>\ldots$ are the real zeros of $\mathrm{Bi}(x)$. The value of $z_{m,0}$ is also modified accordingly (see (\ref{eq77})). This also holds true for the complex zeros of $Y_{\nu}(\nu z)$ using the complex zeros of the $\mathrm{Bi}$ Airy function, but only for a finite number lying in $\Re(z) \geq -1+\delta$ ($\delta>0$); see our discussion on the zeros of the Hankel functions that follows next.

From (\ref{eq59}) the zeros of $H_{\nu}^{(j)}(\nu z)$ ($j=1,2$) that lie in the principal plane with $\Re(z) >-1$ also satisfy the same asymptotic expansions as in \cref{thm3.1}, except that this time
%%%%%%%%%%%%%%%%%%%
\begin{equation}
\label{eq93b}
\zeta_{m,0}
=e^{\mp 2\pi i/3}\nu^{-2/3}\mathrm{a}_{m},
\end{equation}
%%%%%%%%%%%%%%%%%%%
for $j=1,2$ respectively. Consequently in (\ref{eq77}) $|\mathrm{a}_{m}|^{3/2}$ must be replaced by $-|\mathrm{a}_{m}|^{3/2}$. Non-real solutions of this implicit equation must be found numerically, with the branches being such that $\Im(z_{m,0})<0$ for $j=1$ and $\Im(z_{m,0})>0$ for $j=2$.

For large $\nu$ these zeros lie close to the eye-shaped curve $\mathbf{K}$ depicted in \cite[Fig. 10.20.3]{NIST:DLMF}, in the lower half plane for $j=1$, with the conjugate of these for $j=2$. However, in this case $m$ must be truncated at a finite value, as the approximations break down near the second turning point at $z=-1\pm i0$ ($\zeta =e^{\mp \pi i/3}(3\pi/2)^{2/3}$). In particular for each given positive $\nu$ one requires $m$ be bounded via the inequality
%%%%%%%%%%%%%%%%%%%
\begin{equation}
\label{eq93c}
\frac{\left|\mathrm{a}_{m}\right|}{\nu^{2/3}}
\leq \left(\frac{3 \pi}{2}\right)^{2/3}-\delta
\quad (\delta>0).
\end{equation}
%%%%%%%%%%%%%%%%%%%

We can restrict the above described expansions to the half plane $|\arg(z)| \leq \frac12\pi$, and then consider the zeros of $H_{\nu}^{(1)}(\nu z)$ in $-\pi \leq \arg(z) \leq -\frac12\pi$ separately (with the zeros $H_{\nu}^{(2)}(\nu z)$ then obtained by complex conjugating). To this end use \cite[Eq. 10.11.4]{NIST:DLMF}
%%%%%%%%%%%%%%%%%%%
\begin{equation}
\label{eq93d}
H_{\nu}^{(1)}(\nu z e^{-\pi i})
=2\cos(\nu \pi)H_{\nu}^{(1)}(\nu z)
+e^{-\nu \pi i}H_{\nu}^{(2)}(\nu z),
\end{equation}
%%%%%%%%%%%%%%%%%%%
and consider this relation for $0 \leq \arg(z) \leq \frac12\pi$.
% %%%%%%%%%%%%%%%%%%%
% \begin{equation}
% H_{\nu}^{(1)}(\nu z e^{-\pi i})
% =4 \cos(\nu \pi) J_{\nu}(\nu z)
% -e^{\nu \pi i}H_{\nu}^{(2)}(\nu z)
% \end{equation}
% %%%%%%%%%%%%%%%%%%%
Then from (\ref{eq59}) and (\ref{eq93d})
%%%%%%%%%%%%%%%%%%%
\begin{equation}
\label{eq93e}
z^{1/2}H_{\nu}^{(1)}(\nu z e^{-\pi i})
=2\mathcal{Y}(\nu,z)
\widetilde{\mathrm{Ai}}\left(\nu, \nu^{2/3}
\mathcal{Z}(\nu,z)\right),
\end{equation}
%%%%%%%%%%%%%%%%%%%
where
%%%%%%%%%%%%%%%%%%%
\begin{equation}
\label{eq93f}
\widetilde{\mathrm{Ai}}(\nu,t)
=2\cos(\nu \pi)e^{-\pi i/3}
\mathrm{Ai}_{-1}(t)
+e^{-(3\nu-1 )\pi i/3}\mathrm{Ai}_{1}(t).
\end{equation}
%%%%%%%%%%%%%%%%%%%

Let $t=\widetilde{\mathrm{ai}}_{m}(\nu)$ denote the zeros of $\widetilde{\mathrm{Ai}}(\nu,t)$ in the part of the $t$ plane that is the mapping of the quadrant  $0 \leq \arg(z) \leq \frac12\pi$ under $t=\nu^{2/3}\mathcal{Z}(\nu,z)$, enumerated by decreasing values of $\Re(t)$ for $m=1,2,3,\ldots$. In this domain, for large $|t|$ a finite number lie close to the ray $\arg(t) = -\frac13\pi$, with $\widetilde{\mathrm{ai}}_{1}(\nu)$ (the zero with the largest real part) corresponding to the value of $\arg(z)$ closest to, but not exceeding, $\frac12\pi$. Further, if $\cos(\nu \pi) \neq 0$ there are an infinite number of zeros which approach the negative $t$ axis from below as $m \to \infty$; this can be shown from (\ref{eq93f}) and \cite[Eq. 9.7.5]{NIST:DLMF}. See \cite{Gil:2014:OCZ} for more details on the zeros of linear combinations of Airy functions of this kind. 

Then uniform asymptotic expansions for the zeros of $H_{\nu}^{(1)}(\nu z e^{-\pi i})$ lying in the quadrant $0 \leq \arg(z) \leq \frac12\pi$ are again given by (\ref{eq93a}), except now we have $\zeta_{m,0}=\nu^{-2/3}\widetilde{\mathrm{ai}}_{m}(\nu)$. For large $\nu$, and $z$ bounded away from the turning point, a finite number of these zeros lie close to the part of curve $\mathbf{K}$ in the first quadrant, as well as an infinite number that approach the line $\Im(z)=\frac{1}{2}\ln(2)$ as $\Re(z) \to \infty$, provided $\cos(\nu \pi) \neq 0$. As we mentioned, for the zeros of $H_{\nu}^{(2)}(\nu z  e^{\pi i})$ in $-\frac12\pi \leq \arg(z) \leq 0$ we simply take the conjugates of the zeros of $H_{\nu}^{(1)}(\nu z e^{-\pi i})$.

We also have from (\ref{eq60a}), (\ref{eq61}), (\ref{eq62}) and \cite[Eqs. 10.2.3, 10.4.4 and 10.4.6]{NIST:DLMF}
%%%%%%%%%%%%%%%%%%%
\begin{equation}
\label{eq63}
z^{1/2}J_{-\nu}(\nu z)
=\mathcal{Y}(\nu,z) \mathrm{Ai}
\left(\nu, \nu^{2/3}\mathcal{Z}(\nu,z)\right),
\end{equation}
%%%%%%%%%%%%%%%%%%%
where
%%%%%%%%%%%%%%%%%%%
\begin{multline}
\label{eq64}
\mathrm{Ai}(\nu,z)
=\cos(\nu \pi) \mathrm{Ai}(z)
+ \sin(\nu \pi) \mathrm{Bi}(z)
\\
=e^{\pm \nu \pi i}\mathrm{Ai}(z)
+2 e^{\mp \nu \pi i/6} 
\sin(\nu \pi)\mathrm{Ai}_{\pm 1}(z).
\end{multline}
%%%%%%%%%%%%%%%%%%%
For real zeros of $J_{-\nu}(\nu z)$ asymptotic expansions are again provided by a modified version \cref{thm3.1}, in particular using the negative zeros of $\mathrm{Ai}(\nu,z)$ found numerically (see \cite{Gil:2014:OCZ}). For the complex zeros of $\mathrm{Ai}(\nu,z)$ and related Airy functions see \cite{Segura:2013:CCZ}. 

More details on the distribution of the complex zeros of Bessel functions, including the case $\nu$ complex, and for those not lying in the principal plane, are given by Olver \cite{Olver:1954:AEB}.

\section{Zeros of the derivatives of Bessel functions}
\label{sec4}

Expansions for the zeros of the derivatives is done in a similar manner. Let $f(z)$ be the same function as previously, given by (\ref{eq50}), with $g(z)$ replaced by
%%%%%%%%%%%%%%%%%%%
\begin{equation}
\label{eq116}
\tilde{g}(z)=\frac{3z^4+10z^2-1}{4z^{2}
\left(1-z^2\right)^2}.
\end{equation}
%%%%%%%%%%%%%%%%%%%
Then from \cite[Eqs. (4.2) and (4.3)]{Dunster:2014:OEB} we have that
%%%%%%%%%%%%%%%%%%%
\begin{equation}
\label{eq117}
\tilde{w}=z^{3/2}\left(1-z^2\right)^{-1/2}
\mathcal{C}_{\nu}'(\nu z),
\end{equation}
%%%%%%%%%%%%%%%%%%%
where $\mathcal{C}_{\nu}(z)$ is any Bessel function, is a solution of the differential equation
%%%%%%%%%%%%%%%%%%%
\begin{equation}
\label{eq115}
\frac{d^{2} \tilde{w}}{dz^{2}}
=\left\{ \nu ^{2}f(z)+\tilde{g}(z)\right\} \tilde{w}.
\end{equation}
%%%%%%%%%%%%%%%%%%%
Note that $\tilde{g}(z)$ is not analytic at the turning point $z=1$, due to the factor $(1-z^2)^{-1/2}$ in (\ref{eq117}). Thus \cref{thm2.2} cannot be applied directly here, but we proceed in a similar manner, as follows.

Firstly, the Liouville transformation is the same as given by (\ref{eq51}). We derive the coefficients $\tilde{\mathrm{E}}_{s}(\beta)$ (say) in a similar way to  (\ref{eq55}) and (\ref{eq57}), again using \cite[Eqs. (60) - (63)]{Dunster:2022:LAS} but with $g(z)$ replaced by $\tilde{g}(z)$. As a result these are found to be given by
%%%%%%%%%%%%%%%%%%%
\begin{equation}
\label{eq120}
\tilde{\mathrm{E}}_{1}(\beta)
=-\tfrac{1}{24}\beta
\left(7\beta^{2}-9\right),
\quad
\tilde{\mathrm{E}}_{2}(\beta)=
-\tfrac{1}{16}\beta^{2}\left(\beta^{2}-1\right)
\left(7\beta^{2}-3\right),
\end{equation}
%%%%%%%%%%%%%%%%%%%
and for $s=2,3,4,\ldots$
%%%%%%%%%%%%%%%%%%%
\begin{equation}
\label{eq122}
\tilde{\mathrm{E}}_{s+1}(\beta) =
\frac{1}{2} \beta^{2} \left(\beta^{2}-1 \right)\tilde{\mathrm{E}}_{s}^{\prime}(\beta)
+\frac{1}{2}\int_{0}^{\beta}
p^{2}\left(p^{2}-1 \right)
\sum_{j=1}^{s-1}
\tilde{\mathrm{E}}_{j}^{\prime}(p)
\tilde{\mathrm{E}}_{s-j}^{\prime}(p) dp.
\end{equation}
%%%%%%%%%%%%%%%%%%%
where $\beta$ is again defined by (\ref{eq53}). Thus in place of (\ref{eq11}) and (\ref{eq12}) we have the corresponding LG expansions
%%%%%%%%%%%%%%%%%%%
\begin{equation}
\label{eq127}
\tilde{W}_{0}(\nu,z) 
\sim \frac{1}{f^{1/4}(z)}
\exp \left\{ -\nu\xi
+\sum\limits_{s=1}^{\infty}{(-1) ^{s}
\frac{\tilde{\mathrm{E}}_{s}(\beta)}{\nu^{s}}}\right\},
\end{equation}
%%%%%%%%%%%%%%%%%%%
and
%%%%%%%%%%%%%%%%%%%
\begin{equation}
\label{eq128}
\tilde{W}_{\pm 1}(\nu,z)  
\sim \frac{1}{f^{1/4}(z)}
\exp \left\{ \nu\xi
+\sum\limits_{s=1}^{\infty}
\frac{\tilde{\mathrm{E}}_{s}(\beta)}
{\nu^{s}}\right\},
\end{equation}
%%%%%%%%%%%%%%%%%%%
as $\nu \to \infty$, uniformly in the same regions as the corresponding expansions (\ref{eq55}) and (\ref{eq57}).

The identification with the derivatives of the Bessel functions is as follows. Firstly, for the solutions recessive at $z=0$ ($\xi = +\infty$), we have
%%%%%%%%%%%%%%%%%%
\begin{equation}
\label{eq131}
z^{3/2}\left(1-z^2\right)^{-1/2}
J_{\nu}'(\nu z)
= \tilde{C}_{0}(\nu) \tilde{W}_{0}(\nu,z),
\end{equation}
%%%%%%%%%%%%%%%%%%%
where the proportionality constant $\tilde{C}_{0}(\nu)$ can conveniently be found by comparing both sides as $|z| \to \infty$, bearing in mind that in this case $\tilde{\mathrm{E}}_{s}(\beta) \to \tilde{\mathrm{E}}_{s}(0)=0$. To this end, from (\ref{eq51}) as $z \to \pm i \infty$
%%%%%%%%%%%%%%%%%%%
\begin{equation}
\label{eq129}
\xi = \pm iz \mp  \tfrac12\pi i
+\mathcal{O}\left(z^{-1} \right),
\end{equation}
%%%%%%%%%%%%%%%%%%%
and from (\ref{eq50}) $f^{1/4}(z) = e^{ \mp \pi i/4}+ \mathcal{O}(z^{-2})$. Now from \cite[Eq. 10.17.9]{NIST:DLMF}, certainly for $|\arg(z)| \leq \frac12 \pi$, 
%%%%%%%%%%%%%%%%%%%
\begin{equation}
\label{eq130}
J_{\nu}'(\nu z)
= -\left(\frac{2}{\pi \nu z}\right)^{1/2}
\left\{\sin\left(\nu z-\frac{1}{2}\nu \pi-\frac{1}{4}\pi\right)
+\mathcal{O}\left( \frac{1}{z} \right)\right\}
\quad (z \to \infty).
\end{equation}
%%%%%%%%%%%%%%%%%%%

If we let $z \to i \infty$ (or $z \to -i \infty$) then from (\ref{eq127}), (\ref{eq131}), (\ref{eq129}) and (\ref{eq130}) we arrive at $\tilde{C}_{0}(\nu)=1/\sqrt{2 \pi \nu}$, and hence from (\ref{eq50}), (\ref{eq127}) and (\ref{eq131})
%%%%%%%%%%%%%%%%%%%
\begin{equation}
\label{eq132}
J_{\nu}'(\nu z)
\sim 
\frac{\left(1-z^2\right)^{1/4}}
{\sqrt{2 \pi \nu} \, z}
\exp \left\{ -\nu\xi
+\sum\limits_{s=1}^{\infty}{(-1) ^{s}
\frac{\tilde{\mathrm{E}}_{s}(\beta)}{\nu^{s}}}\right\},
\end{equation}
%%%%%%%%%%%%%%%%%%%
as $\nu \to \infty$, uniformly for $z \in Z_{0}$.

In a similar manner for solutions recessive at $z=i\infty$ (see \cite[Eq. 10.17.11]{NIST:DLMF}) one can show that as $\nu \to \infty$
%%%%%%%%%%%%%%%%%%%
\begin{equation}
\label{eq134}
{H_{\nu}^{(1)}}'(\nu z)
\sim i \sqrt{\frac{2}{\pi \nu}}
\frac{\left(1-z^2\right)^{1/4}}{z}
\exp \left\{ \nu\xi
+\sum\limits_{s=1}^{\infty}
\frac{\tilde{\mathrm{E}}_{s}(\beta)}
{\nu^{s}}\right\},
\end{equation}
%%%%%%%%%%%%%%%%%%%
% Next as $z \to i \infty$ from 
% %%%%%%%%%%%%%%%%%%%
% \begin{equation}
% \label{eq133}
% {H_{\nu}^{(1)}}'(\nu z)
% = i\left(\frac{2}{\pi \nu z}\right)^{1/2}
% \exp \left\{i\left(\nu z-\frac{1}{2}\nu \pi
% -\frac{1}{4}\pi\right)\right\}
% \left\{1+
% \mathcal{O}\left( \frac{1}{z} \right)\right\}
% \end{equation}
% %%%%%%%%%%%%%%%%%%%
uniformly for $z \in Z_{-1}$, and likewise
%%%%%%%%%%%%%%%%%%%
\begin{equation}
\label{eq135}
{H_{\nu}^{(2)}}'(\nu z)
\sim -i \sqrt{\frac{2}{\pi \nu}}
\frac{\left(1-z^2\right)^{1/4}}{z}
\exp \left\{ \nu\xi
+\sum\limits_{s=1}^{\infty}
\frac{\tilde{\mathrm{E}}_{s}(\beta)}
{\nu^{s}}\right\},
\end{equation}
%%%%%%%%%%%%%%%%%%%
uniformly for $z \in Z_{1}$.

% Crossing the cut $[1,\infty)$ yields in H2 region (with principal value of $\xi$ so that $\Re(\xi) \leq 0$) we have $\xi \to -\xi$, $\beta \to -\beta$ and $(1-z^2)^{1/4} \to -i(1-z^2)^{1/4}$
% %%%%%%%%%%%%%%%%%%%
% \begin{equation}
% \label{eq136}
% {H_{\nu}^{(1)}}'(\nu z)
% \sim 
% \sqrt{\frac{2}{\pi \nu}}
% \frac{\left(1-z^2\right)^{1/4}}{z}
% \exp \left\{ -\nu\xi
% +\sum\limits_{s=1}^{\infty}{(-1) ^{s}
% \frac{\tilde{\mathrm{E}}_{s}(\beta)}{\nu^{s}}}\right\}
% \end{equation}
% %%%%%%%%%%%%%%%%%%%

Based on (\ref{eq59}) we define $\tilde{\mathcal{Y}}(\nu,z)$ and $\tilde{\mathcal{Z}}(\nu,z)$ via the pair of equations
%%%%%%%%%%%%%%%%%%%
\begin{equation}
\label{eq123}
{H_{\nu}^{(1,2)}}'(\nu z)
=2 e^{\mp \pi i/3}
\tilde{\mathcal{Y}}(\nu,z)
\mathrm{Ai}_{\mp 1}^{\prime}
\left( \nu^{2/3}
\tilde{\mathcal{Z}}(\nu,z)\right).
\end{equation}
%%%%%%%%%%%%%%%%%%%
Using these and the derivative of the Airy connection formula (\ref{eq16}) we then get
%%%%%%%%%%%%%%%%%%%
\begin{equation}
\label{eq126}
J_{\nu}'(\nu z)
=\tilde{\mathcal{Y}}(\nu,z)
\mathrm{Ai}^{\prime}
\left( \nu^{2/3}
\tilde{\mathcal{Z}}(\nu,z)\right).
\end{equation}
%%%%%%%%%%%%%%%%%%%
Compare this with (\ref{eq61}) as well as \cite[Eq. (4.15)]{Dunster:2014:OEB}.

Again our main focus is the asymptotic expansion of $\tilde{\mathcal{Z}}(\nu,z)$, since this will yield the expansions for the zeros. We can again eliminate $\tilde{\mathcal{Y}}(\nu,z)$ in a similar manner, by dividing any pair from the three equations (\ref{eq123}) and (\ref{eq126}). For example
%%%%%%%%%%%%%%%%%%%
\begin{equation}
\label{eq137}
\frac{{H_{\nu}^{(1)}}'(\nu z)}{J_{\nu}'(\nu z)}
=\frac{2 e^{-\pi i/3}
\mathrm{Ai}_{-1}^{\prime}
\left( \nu^{2/3}
\tilde{\mathcal{Z}}(\nu,z)\right)}
{\mathrm{Ai}^{\prime}
\left( \nu^{2/3}
\tilde{\mathcal{Z}}(\nu,z)\right)}.
\end{equation}
%%%%%%%%%%%%%%%%%%%

We require expansions for the derivatives of Airy functions, of the form (\ref{eq25}) and (\ref{eq26}). These involve the sequence $\{\tilde{a}_{s}\}_{s=1}^{\infty}$ defined recursively by $\tilde{a}_{1}=\tilde{a}_{2}=-\frac{7}{72}$, and subsequent terms $\tilde{{a}}_{s}$ ($s=3,4,5,\ldots $) from the formula (\ref{eq24}) with $a$ replaced by $\tilde{a}$. Thus let
%%%%%%%%%%%%%%%%%%%
\begin{equation}
\label{eq141}
\tilde{\mathcal{X}}
=\tilde{\mathcal{X}}(\nu,z) 
=\tfrac{2}{3}\tilde{\mathcal{Z}}^{3/2}(\nu,z),
\end{equation}
%%%%%%%%%%%%%%%%%%%
and then from \cite[Thm. 2.4]{Dunster:2021:SEB} as $\nu \tilde{\mathcal{X}} \to \infty$ for $\delta>0$
%%%%%%%%%%%%%%%%%%%
\begin{equation}
\label{eq138}
\mathrm{Ai}^{\prime }\left(\nu^{2/3}\tilde{\mathcal{Z}}\right) 
\sim  -\frac{\nu^{1/6}\tilde{\mathcal{Z}} ^{1/4}}
{2\pi ^{1/2}}\exp \left\{-\nu\tilde{\mathcal{X}} 
+\sum\limits_{s=1}^{\infty}{(-1)^{s}
\frac{\tilde{a}_{s}}
{s (\nu\tilde{\mathcal{X}})^{s}}}\right\}
\quad \left(|\arg(\tilde{\mathcal{Z}})| \leq \pi-\delta \right),
\end{equation}
%%%%%%%%%%%%%%%%%%%
and 
%%%%%%%%%%%%%%%%%%%
\begin{multline}
\label{eq139}
\mathrm{Ai}_{\pm 1}^{\prime }
\left(\nu^{2/3}\tilde{\mathcal{Z}} \right) 
\sim \frac{e^{\pm \pi i/6}\nu^{1/6}
\tilde{\mathcal{Z}} ^{1/4}}{2\pi ^{1/2}}
\exp \left\{ \nu\tilde{\mathcal{X}}
+\sum\limits_{s=1}^{\infty}
\frac{\tilde{a}_{s}}{s 
(\nu\tilde{\mathcal{X}})^{s}}\right\}
\\
\quad \left(|\arg(\tilde{\mathcal{Z}} 
e^{\mp 2\pi i/3})| \leq \pi-\delta \right).
\end{multline}
%%%%%%%%%%%%%%%%%%%
% Consider the $z$ region $Z_{0} \cap Z_{-1}$ where $-\pi \leq \arg(\xi) \leq 0$ ($-\frac23\pi \leq \arg(\zeta) \leq 0$). In this $J_{\nu}'(\nu z)$ has no zeros and the LG expansions for J' and H1' are valid.
Then similarly to (\ref{eq31}), using (\ref{eq132}) - (\ref{eq139}), one obtains
%%%%%%%%%%%%%%%%%%%
\begin{equation}
\label{eq140}
\nu\left\{\tilde{\mathcal{X}}(\nu,z)-\xi \right\}
+\sum\limits_{s=0}^{\infty}\frac{
\tilde{a}_{2s+1}}{(2s+1) 
\left\{\nu \tilde{\mathcal{X}}(\nu,z)\right\}^{2s+1}}
-\sum\limits_{s=0}^{\infty}
\frac{\tilde{\mathrm{E}}_{2s+1}(\beta)}{\nu^{2s+1}}
\sim  0.
\end{equation}
%%%%%%%%%%%%%%%%%%%

Consequently the expansion for $\tilde{\mathcal{Z}}(\nu,z)$ is of the same form as (\ref{eq32a}), namely
%%%%%%%%%%%%%%%%%%%
\begin{equation}
\label{eq142}
\tilde{\mathcal{Z}}(\nu,z)  
\sim  \zeta
+\sum_{s=1}^{\infty}
\frac{\tilde{\Upsilon}_{s}(z)}{\nu^{2s}}.
\end{equation}
%%%%%%%%%%%%%%%%%%%
The coefficients $\tilde{\Upsilon}_{s}(z)$ here are determined in the same manner as \cref{thm2.2}, but with the LG coefficients $\hat{\mathrm{E}}_{s}(\beta)$ replaced by $\tilde{\mathrm{E}}_{s}(\beta)$ (given by (\ref{eq120}) and (\ref{eq122})), and in (\ref{eq38}) $a_{2k+1}$ replaced by $\tilde{a}_{2k+1}$. 

As a result we find for the first four
%%%%%%%%%%%%%%%%%%%
\begin{equation}
\label{eq143}
\tilde{\Upsilon}_{1}(z)
=-\frac{1}{ 48\zeta^{2}}
\left\{14 \sigma^{3}-18\sigma\zeta-7
\right\},
\end{equation}
%%%%%%%%%%%%%%%%%%%
%%%%%%%%%%%%%%%%%%%
\begin{multline}
\label{eq144}
\tilde{\Upsilon}_{2}(z)
=-\frac{1}{ 11520\zeta^{5}}
\left\{
14630\sigma^{9}-28350\sigma^{7}\zeta
+15642\sigma^{5}\zeta^{2}
+245\sigma^{6}
-630\sigma^{4}\zeta
\right.
\\
\left.
+405\sigma^2\zeta^2
-70 \left( 27\zeta^{3}+14 \right) \sigma^{3}
+1260\sigma\zeta-1400
\right\},
\end{multline}
%%%%%%%%%%%%%%%%%%%
%%%%%%%%%%%%%%%%%%%
\begin{multline}
\label{eq145}
\tilde{\Upsilon}_{3}(z)
=-\frac{1}{5806080 \zeta^{8}}
\left\{
187212438\sigma^{15}-574909650\sigma^{13}\zeta
+643669740{\sigma}^{11}\zeta^{2}
\right.
\\
+1075305\sigma^{12}-3466260\sigma^{10}\zeta
+3828762\sigma^{8}\zeta^{2}
-700 \left( 445923\zeta^{3}+3038 \right)\sigma^{9}
\\
+270 \left( 216009\zeta^{3}+15092 \right)\sigma^{7}\zeta
-378\left( 5697\zeta^{3}+5768\right)\sigma^{5}\zeta^{2}
\\
-21\left( 77004\zeta^{3}+12005\right) \sigma^{6}
 +6615 \left( 27\zeta^{3}+98\right)\sigma^{4}\zeta
 -416745\sigma^{2}\zeta^{2}
\\
\left.
+4200 \left( 54\zeta^{3}-245 \right) \sigma^{3}
+1323000\sigma\zeta-5074244
\right\},
\end{multline}
%%%%%%%%%%%%%%%%%%%
and
%%%%%%%%%%%%%%%%%%%
\begin{multline}
\label{eq146}
\tilde{\Upsilon}_{4}(z)
=-\frac{1}{2786918400 \zeta^{11}}
\left\{
5855137662150\sigma^{21}
-24861164091750\sigma^{19}\zeta
\right.
\\
+42460124455350\sigma^{17}\zeta^{2}
+14228564385\sigma^{18}
-61447780170\sigma^{16}\zeta
+103421139255\sigma^{14}\zeta^{2}
\\
-210 \left( 176374197387\zeta^{3}+124091422 \right) \sigma^{15}
+9450 \left( 1823471757\zeta^{3}
+8445346 \right) \sigma^{13}\zeta
\\
-1890 \left( 2126608065\zeta^{3}+47065508 \right) \sigma^
{11}\zeta^{2}
-175 \left( 484154676\zeta^{3}+6004901 \right)\sigma^{12}
\\
+63 \left( 540045297\zeta^{3}
+53679500 \right)\sigma^{10}\zeta
-270 \left( 21148803\zeta^{3}+13788943 \right) 
\sigma^{8}\zeta^{2}
\\
+70 \left( 5408136315\zeta^{6}
+611212644\zeta^{3}-31563350 \right) \sigma^{9}
\\
-270 \left( 24456735{\zeta}^{6}+29070594\zeta^{3}
-16275350 \right) \sigma^{7}\zeta
+26460 \left( 10179\zeta^{3}-97925 \right) 
\sigma^{5}\zeta^{2}
\\
+105 \left( 2024433\zeta^{6}+14854644\zeta^{3}
-4459000\right) \sigma^{6}
-33075 \left( 5049\zeta^{3}-36400 \right) 
\sigma^{4}\zeta
\\
\left.
-773955000\sigma^{2}\zeta^{2}
+6860 \left(58725\zeta^{3}-828448 \right) \sigma^{3}
+7306911360\sigma\zeta-41850751040
\right\},
\end{multline}
%%%%%%%%%%%%%%%%%%%
where $\sigma$ is again defined by (\ref{eq65}).

Next from Airy's equation \cite[Eq. 9.2.1]{NIST:DLMF} a simple calculation gives
%%%%%%%%%%%%%%%%%%%
\begin{equation}
\label{eq22a}
\frac{d}{dt} \left\{
\frac{\mathrm{Ai}_{j}^{\prime}(t)}
{\mathrm{Ai}_{k}^{\prime}(t)}\right\}
=\frac{t \mathscr{W}\left\{\mathrm{Ai}_{j}(t)
\mathrm{Ai}_{k}(t)\right\}}
{\left\{\mathrm{Ai}_{k}^{\prime}(t)\right\}^{2}},
\end{equation}
%%%%%%%%%%%%%%%%%%%
which in contrast to (\ref{eq22}) vanishes at $t=0$. Now from (\ref{eq137}) and \cite[Eqs. 10.19.12 and 10.19.13]{NIST:DLMF} for large $\nu$ we have $z \approx 1$ when $\tilde{\mathcal{Z}}(\nu,z)=0$, and hence \cite[Thm. 2.1]{Fabijonas:1999:ORA} does not guarantee that $\tilde{\mathcal{Z}}(\nu,z)$, as given implicitly by (\ref{eq137}), is analytic near $z=1$. Indeed as $z \to 1$ ($\zeta \to 0$) we find from (\ref{eq51}), (\ref{eq65}) and (\ref{eq143})
%%%%%%%%%%%%%%%%%%%
\begin{equation}
\label{eq147}
\tilde{\Upsilon}_{1}(z)
\sim \frac{2^{1/3}}{10(1-z)},
\end{equation}
%%%%%%%%%%%%%%%%%%%
with the other coefficients also having higher order poles at $z=1$, unlike those for $\mathcal{Z}(\nu,z)$ (cf. (\ref{eq71})- (\ref{eq74})).  We also anticipate lack of analyticity near $z=1$ from (\ref{eq116}), and indeed this was discussed in \cite{Dunster:2014:OEB}. As we shall see this is not a major problem, as the zeros will not be too close to $z=1$; see (\ref{eq106a}) - (\ref{eq109a}) below and the discussion proceeding them.

% From (\ref{eq53}), (\ref{eq120}), (\ref{eq122}) we have $\tilde{\mathrm{E}}_{s}(\beta)=\mathcal{O}\{(z-1)^{-3s/2}\}$ as $z \to 1$.

% From \cite[Eqs. 1.9.60 - 1.9.62]{NIST:DLMF} we have by induction from (\ref{eq140})- (\ref{eq142}) $\tilde{\Upsilon}_{s}(z)=\mathcal{O}\{(1-z)^{1-2s}\}$

Now let us consider the zeros $\tilde{z}_{m}$ ($m=1,2,3,\ldots$) which satisfy
%%%%%%%%%%%%%%%%%%%
\begin{equation}
\label{eq148}
J_{\nu}'\left(\nu \tilde{z}_{m} \right) = 0.
\end{equation}
%%%%%%%%%%%%%%%%%%%
From (\ref{eq126}) we have 
%%%%%%%%%%%%%%%%%%%
\begin{equation}
\label{eq149}
\mathcal{Z}(\nu,\tilde{z}_{m})
= \nu^{-2/3}{\mathrm{a}}^{\prime}_{m},
\end{equation}
%%%%%%%%%%%%%%%%%%%
where ${\mathrm{a}}^{\prime}_{m}$ ($m=1,2,3,\ldots$) are the negative zeros of $\mathrm{Ai}^{\prime}(x)$ ordered by increasing absolute values; cf. (\ref{eq82a}).

Similarly to the non-derivative zeros of the previous section we have for large $\nu$ and any positive integer $m$ the first order approximation $\tilde{z}_{m} \approx  \tilde{z}_{m,0}$, where this time $\tilde{z}_{m,0}$ satisfies
%%%%%%%%%%%%%%%%%%%
\begin{equation}
\label{eq150}
\zeta\left(\tilde{z}_{m,0}\right)
= \nu^{-2/3}{\mathrm{a}}^{\prime}_{m} := \tilde{\zeta}_{m,0},
\end{equation}
%%%%%%%%%%%%%%%%%%%
and so in place of (\ref{eq77})
%%%%%%%%%%%%%%%%%%%
\begin{equation}
\label{eq151}
\int_{1}^{\tilde{z}_{m,0}}
\frac{\left(t^{2}-1\right)^{1/2}}{t}\,dt
=\sqrt{\tilde{z}_{m,0}^{2}-1}
-\mathrm{arcsec}\left( \tilde{z}_{m,0} \right)
=\frac{2}{3 \nu}
\left|{\mathrm{a}}^{\prime}_{m}\right|^{3/2}.
\end{equation}
%%%%%%%%%%%%%%%%%%%

The derivation of the asymptotic expansion for the zeros in the present case is analogous to \cref{sec3}, and we arrive at
%%%%%%%%%%%%%%%%%%%
\begin{equation}
\label{eq155}
\tilde{z}_{m} \sim \tilde{z}_{m,0}
+\sum_{s=1}^{\infty}
\frac{\tilde{z}_{m,s}}{\nu^{2s}}
\quad (\nu \to \infty),
\end{equation}
%%%%%%%%%%%%%%%%%%%
where $\tilde{z}_{m,s}$ ($s=1,2,3,\ldots$) are determined the same way as $z_{m,s}$ (see \cref{thm3.1}), but with the differing initial values $\tilde{z}_{m,0}$ and $\tilde{\zeta}_{m,0}$ as given above, and the functions $\Upsilon_{s}(z)$ replaced by $\tilde{\Upsilon}_{s}(z)$. We remark that Olver \cite[\S 7]{Olver:1954:AEB} demonstrated the existence of such an expansion.

Thus if we denote $\tilde{\sigma}_{m,0}=\sigma(\tilde{z}_{m,0})$,
%%%%%%%%%%%%%%%%%%%
\begin{equation}
\label{eq153}
\tilde{\zeta}_{m,0}^{\prime}
=\zeta'\left(\tilde{z}_{m,0}\right),  \;
\tilde{\zeta}_{m,0}^{\prime \prime}
=\zeta''\left(\tilde{z}_{m,0}\right),  \;
\ldots,
\end{equation}
%%%%%%%%%%%%%%%%%%%
and for $s=1,2,3,\ldots$
%%%%%%%%%%%%%%%%%%%
\begin{equation}
\label{eq154}
\tilde{\Upsilon}_{m,s}
=\tilde{\Upsilon}_{s}\left(\tilde{z}_{m,0}\right), \;
\tilde{\Upsilon}_{m,s}^{\prime}
=\tilde{\Upsilon}_{s}'\left(\tilde{z}_{m,0}\right),  \;
\tilde{\Upsilon}_{m,s}^{\prime \prime}
=\tilde{\Upsilon}_{s}''\left(\tilde{z}_{m,0}\right),  \;
\ldots,
\end{equation}
%%%%%%%%%%%%%%%%%%%
we find that $\tilde{z}_{m,s}$ ($s=1,2,3,4$) are also given by (\ref{eq94}) - (\ref{eq97}), but with $z$, $\zeta$ and $\Upsilon$ replaced by $\tilde{z}$, $\tilde{\zeta}$ and $\tilde{\Upsilon}$, respectively. Therefore using (\ref{eq143}) and (\ref{eq144}) we have for the first two
%%%%%%%%%%%%%%%%%%%
\begin{equation}
\label{eq156a}
\tilde{z}_{m,1}=
-\frac{\tilde{z}_{m,0} \tilde{\sigma}_{m,0}}{48  
\tilde{\zeta}_{m,0}^{2}}
\left\{ 14\tilde{\sigma}_{m,0}^3
- 18\tilde{\sigma}_{m,0} \tilde{\zeta}_{m,0} - 7 \right\},
\end{equation}
%%%%%%%%%%%%%%%%%%%
and
%%%%%%%%%%%%%%%%%%%
\begin{multline}
\label{eq157a}
\tilde{z}_{m,2}=
\frac{\tilde{z}_{m,0} \,\tilde{\sigma}_{m,0}}{46080
\tilde{\zeta}_{m,0}^{5}}
\left\{
280\tilde{\sigma}_{m,0}^{9}\left(49\tilde{z}_{m,0}^{2}
-209\right) 
-560\tilde{\sigma}_{m,0}^{7} \, \tilde{\zeta}_{m,0}\left( 45\tilde{z}_{m,0}^{2}-206\right) 
\right.
\\
-7840\tilde{\sigma}_{m,0}^{6} \, \tilde{z}_{m,0}^{2}
+216\tilde{\sigma}_{m,0}^{5} \, \tilde{\zeta}_{m,0}^{2}\left( 45 \tilde{z}_{m,0}^{2}-313\right)
+280\tilde{\sigma}_{m,0}^{4} \, \tilde{\zeta}_{m,0}
\left(18 \tilde{z}_{m,0}^{2}-7 \right) 
 \\
 \left.
+2520\tilde{\sigma}_{m,0}^{2} \, \tilde{\zeta}_{m,0}^{2}
+10 \tilde{\sigma}_{m,0}^{3}
\left( 49\tilde{z}_{m,0}^{2} 
+ 1080\tilde{\zeta}_{m,0}^{3}\right) 
+490\tilde{\sigma}_{m,0} \,\tilde{\zeta}_{m,0}+7315
\right\},
\end{multline}
%%%%%%%%%%%%%%%%%%%
which can be compared to (\ref{eq103}) and (\ref{eq105}). The next two are given in the second Maple worksheet of \cref{secA}.

In contrast to (\ref{eq106}) - (\ref{eq109}), we find as $\tilde{z}_{m,0} \to 1$ ($\tilde{\zeta}_{m,0} \to 0$)
%%%%%%%%%%%%%%%%%%%
\begin{equation}
\label{eq106a}
\tilde{z}_{m,1}=
-\tfrac{1}{10} \left(\tilde{z}_{m,0}-1\right)^{-1}
\left\{1
+\mathcal{O}\left(\tilde{z}_{m,0}-1\right)\right\},
\end{equation}
%%%%%%%%%%%%%%%%%%%
%%%%%%%%%%%%%%%%%%%
\begin{equation}
\label{eq107a}
\tilde{z}_{m,2}=
-\tfrac{1}{200} \left(\tilde{z}_{m,0}-1\right)^{-3}
\left\{1
+\mathcal{O}\left(\tilde{z}_{m,0}-1\right)\right\},
\end{equation}
%%%%%%%%%%%%%%%%%%%
%%%%%%%%%%%%%%%%%%%
\begin{equation}
\label{eq108a}
\tilde{z}_{m,3}=
-\tfrac{1}{2000} \left(\tilde{z}_{m,0}-1\right)^{-5}
\left\{1
+\mathcal{O}\left(\tilde{z}_{m,0}-1\right)\right\},
\end{equation}
%%%%%%%%%%%%%%%%%%%
and
%%%%%%%%%%%%%%%%%%%
\begin{equation}
\label{eq109a}
\tilde{z}_{m,4}=
-\tfrac{1}{16000} \left(\tilde{z}_{m,0}-1\right)^{-7}
\left\{1
+\mathcal{O}\left(\tilde{z}_{m,0}-1\right)\right\}.
\end{equation}
%%%%%%%%%%%%%%%%%%%

On the other hand, one can derive sharp estimates similar to (\ref{eq106}) - (\ref{eq106}) as $\tilde{z}_{m,0} \to \infty$, namely $\tilde{z}_{m,s} \sim (-1)^{s}k_{s}\{z_{m,0}\}^{-2s+1}$ ($s=1,2,3,4$), where
%%%%%%%%%%%%%%%%%%%
\begin{equation}
\label{eq109b}
k_{1} = \tfrac{5}{18}, \quad
k_{2} = \tfrac{25}{1944}, \quad
k_{3} = \tfrac{1465}{11664}, \quad
k_{4} = \tfrac{5354165}{5878656}.
\end{equation}
%%%%%%%%%%%%%%%%%%%

Now $\tilde{z}_{m,0} \to 1$ when $\nu \to \infty$ with $m$ bounded. In this case from (\ref{eq150}) and (\ref{eq151}) it is straightforward to verify that $(\tilde{z}_{m}-1)^{-1}$ is $\mathcal{O}(\nu^{2/3})$ (cf. (\ref{eq76}) and (\ref{eq78})). Hence in this ``worst'' case the first four terms in the series of (\ref{eq155}) are $\mathcal{O}(\nu^{-2(s+1)/3})$, as opposed to being $\mathcal{O}(\nu^{-2s})$ when $\tilde{z}_{m,0}$ is not close to $1$. With this observation, and taking in account the small leading constants in (\ref{eq106a}) - (\ref{eq109a}), we can still expect good accuracy, but not at the same level as for the approximations in the previous section for comparable small $m$ and large $\nu$ values. This will be illustrated in the next section.

We finally remark that expansions for the derivatives of the other Bessel functions can be derived in the same way as described in \cref{sec3.1}.

\section{Numerical examples}
\label{sec5}

We compute scaled versions of $J_{\nu}(\nu z)$ and $J_{\nu}'(\nu z)$ to give a more precise estimate of the error of our approximations to their zeros. Thus from \cite[Eqs. 10.19.6 and 10.19.7]{NIST:DLMF} let us define
%%%%%%%%%%%%%%%%%%%
\begin{equation}
\label{eq160}
\mathbb{J}(\nu,z):=
\sqrt{\frac{\pi \nu}{2}}
\left(z^2-1\right)^{1/4}
J_{\nu}(\nu z)
=\cos\left(\nu \xi -\tfrac{1}{4}\pi\right)
+\mathcal{O}\left(\frac{1}{\nu z}\right),
\end{equation}
%%%%%%%%%%%%%%%%%%%
and
%%%%%%%%%%%%%%%%%%%
\begin{equation}
\label{eq161}
\tilde{\mathbb{J}}(\nu,z):=
\sqrt{\frac{\pi \nu}{2}}
\frac{z J_{\nu}'(\nu z)}{\left(z^2-1\right)^{1/4}}
=-\sin\left(\nu \xi -\tfrac{1}{4}\pi\right)
+\mathcal{O}\left(\frac{1}{\nu z}\right),
\end{equation}
%%%%%%%%%%%%%%%%%%%
with $\xi$ being given by (\ref{eq51}), and where the $O$ terms hold uniformly for $\nu \to \infty$ and $1+\delta \leq z < \infty$ ($\delta>0$). In (\ref{eq161}) we compute the derivative via the well-known relation \cite[Eq. 10.6.2]{NIST:DLMF}
%%%%%%%%%%%%%%%%%%%
\begin{equation}
\label{eq162}
J_{\nu}'(x)
=(\nu/x)J_{\nu}(x)
-J_{\nu+1}(x)
\quad (x \neq 0).
\end{equation}
%%%%%%%%%%%%%%%%%%%

Taking the first five terms in (\ref{eq93a}) we have for unrestricted $m=1,2,3,\ldots$ the uniform asymptotic approximation for one or both of $\nu$ and $m$ large
%%%%%%%%%%%%%%%%%%%
\begin{equation}
\label{eq163}
z_{m} \approx 
z_{4}(\nu,m)
:= z_{m,0}+\frac{z_{m,1}}{\nu^{2}}
+\frac{z_{m,2}}{\nu^{4}}
+\frac{z_{m,3}}{\nu^{6}}
+\frac{z_{m,4}}{\nu^{8}}.
\end{equation}
%%%%%%%%%%%%%%%%%%%
In \cref{table:1} values of $\mathbb{J}(\nu,z_{4}(\nu,m))$ for various values of $\nu$ and $m$ are tabulated. We see very good approximation for all values of the parameters except $\nu=m=1$, and this exception is not surprising since our uniform approximations are valid for one or both of $\nu$ and $m$ large. Nevertheless, even in this case we do observe about 6 digits of accuracy.

We remark that the Airy zeros used in (\ref{eq77}) and (\ref{eq81}) to evaluate $z_{m,0}$ were computed with the Maple command AiryAiZeros(m). However to speed up computation even further these values could be precomputed and stored \cite{Fabijonas:2004:CCA}. Use of an asymptotic expansion \cite{Fabijonas:1999:ORA} is also efficient for large values of $m$, and in particular for complex zeros (such as those used for the Hankel function zeros).

%\vspace{5pt}

\begin{table}[ht]
\centering
\begin{adjustbox}{width=1\textwidth}
\begin{tabular}{|c|| c |c| c| c| c|} 
 \hline
 \\[-1em]
$\mathbb{J}(\nu,z_{4}(\nu,m))$ 
& $m=1$ & $m=5$ & $m=10$ & $m=100$ & $m=1000$ \\ [0.5ex] 
 \hline\hline
 \\[-1em]
  $\nu=1$ 
  & 1.5166 e-06
  & 5.5411 e-11 
  &-2.0881 e-13 
  & -3.8278 e-22 
  & -4.0685 e-31\\
   \hline
   \\[-1em]
   $\nu=5$ 
   & 4.8691 e-11
   & 4.8816 e-13 
   & -1.2090 e-14 
   & -2.6932 e-22  
   & -3.9246 e-31 \\
   \hline
   \\[-1em]
   $\nu=10$ 
   & 1.6998 e-13 
   & 1.1644 e-14 
   & -8.6674 e-16 
   & -1.7677 e-22  
   & -3.7526 e-31 \\
   \hline
   \\[-1em]
   $\nu=100$ 
   & 1.9854 e-22 
   & 2.0479 e-22 
   & -1.4881 e-22 
   & -7.8795 e-25  
   & -1.7358 e-31 \\
   \hline
   \\[-1em]
 $\nu=1000$ 
 & 1.0908 e-31 
 & 1.7930 e-31 
 & -2.0627 e-31 
 & -1.4651 e-31  
 & -7.8009 e-36 \\ [1ex] 
 \hline
\end{tabular}
\end{adjustbox}
\vspace{5pt}
\caption{$\mathbb{J}(\nu,z)$ computed at the approximate zero $z=z_{4}(\nu,m)$ for various values of $\nu$ and $m$.}
\label{table:1}
\end{table}

As a further check on the Maple Bessel function routines, and in particular our choice of the scaling factors, we found for all values of $\nu$ and $m$ in \cref{table:1} that $\tilde{\mathbb{J}}_{\nu}(z_{4}(\nu,m))\approx (-1)^{m}$ as expected from (\ref{eq161}) and the known alternating sign of $J_{\nu}'(\nu z_{m})$. More precisely we found that $|\tilde{\mathbb{J}}_{\nu}(z_{4}(\nu,m))-(-1)^{m}|<0.0057$. Note that we certainly do not expect that this difference be very small.

%\vspace{5pt}

\begin{table}[ht]
\centering
\begin{adjustbox}{width=1\textwidth}
\begin{tabular}{|c|| c |c| c| c| c|} 
 \hline
 \\[-1em]
$\tilde{\mathbb{J}}(\nu,\tilde{z}_{4}(\nu,m))$ 
& $m=1$ & $m=5$ & $m=10$ & $m=100$ & $m=1000$ \\ [0.5ex] 
 \hline\hline
 \\[-1em]
  $\nu=1$ 
  &  -6.1638 e-05 
  &  -1.3266 e-10 
  &  3.3843 e-13 
  &  4.2065 e-22 
  &  4.2923 e-31\\
   \hline
   \\[-1em]
   $\nu=5$ 
   &  -2.5791 e-07
   &  -8.6953 e-13 
   &  1.7800 e-14 
   &  2.9612 e-22  
   &  4.1413 e-31 \\
   \hline
   \\[-1em]
   $\nu=10$ 
   &  -2.8314 e-08 
   &  -1.8659 e-14 
   &   1.2075 e-15 
   &   1.9451 e-22  
   &   3.9608 e-31 \\
   \hline
   \\[-1em]
   $\nu=100$ 
   &  -1.4954 e-11 
   &  -5.9170 e-19 
   &  5.9462 e-21 
   &  8.8919 e-25  
   &  1.8405 e-31 \\
   \hline
   \\[-1em]
 $\nu=1000$ 
 &  -7.1291 e-15 
 &  -3.4984 e-22 
 &   4.0750 e-24 
 &  3.9395 e-30
 &  8.6458 e-34 \\ [1ex] 
 \hline
\end{tabular}
\end{adjustbox}
\vspace{5pt}
\caption{$\tilde{\mathbb{J}}(\nu,z)$ computed at the approximate zero $z=\tilde{z}_{4}(\nu,m)$ for various values of $\nu$ and $m$.}
\label{table:2}
\end{table}

In \cref{table:2} computed values of $\tilde{\mathbb{J}}(\nu,\tilde{z}_{4}(\nu,m))$ are given for various values of $\nu$ and $m$. Here $\tilde{z}_{4}(\nu,m)$, our approximation to the $m$th exact zero $\tilde{z}_{m}$ of $J_{\nu}^{\prime}(\nu z)$, is given by (\ref{eq163}) with $z_{m,s}$ replaced by $\tilde{z}_{m,s}$ ($s=0,1,2,3,4$). As remarked at the end of \cref{sec4} as expected we observe less accurate values when $\nu$ is large and $m$ is small as compared to \cref{table:1}, but not significantly so. We see this when comparing the $m=1$ columns of both tables, and to a lesser extent the $m=5$ columns.

In evaluating the leading term $\tilde{z}_{m,0}$ via (\ref{eq151}) we solved numerically for the zeros of the derivative of the Airy function, since Maple does not have a built in routine for this. Now from \cite{Pittaluga:1991:IZA} it was shown that the asymptotic expansions for the zeros of $\mathrm{Ai}(x)$ \cite[Eqs. 9.9.6 and 9.9.18]{NIST:DLMF} when truncated after one to five of the leading terms provide upper and lower bounds for the zeros (Nemes \cite{Nemes:2021:PTC} extended this to all terms). From this, and the interlacing of the zeros of $\mathrm{Ai}(x)$ and $\mathrm{Ai}'(x)$, we have
%%%%%%%%%%%%%%%%%%%
\begin{equation}
\label{eq164}
-t_{m}^{2/3}\left(1
+\tfrac{5}{48}t_{m}^{-2}\right)
< \mathrm{a}_{m}
<{\mathrm{a}}^{\prime}_{m}
<\mathrm{a}_{m-1}
< -t_{m-1}^{2/3}
\quad (m=2,3,4,\ldots),
\end{equation}
%%%%%%%%%%%%%%%%%%%
where $t_{m}=\frac38(4m-1)\pi$. Consequently, on including $\mathrm{a}_{1}<{\mathrm{a}}^{\prime}_{1}<0$, we were able to readily numerically solve for the $m$th zero of $\mathrm{Ai}'(x)$ in the interval $(x_{1,m},x_{2,m})$ where
% %%%%%%%%%%%%%%%%%%%
% \begin{multline}
% \label{eq165a}
% -\frac{\{3(4m-1)\pi\}^{2/3}}{4}
% \left(1+\frac{20}{27\{(4m-1)\pi\}^{2}}\right)
% < x
% \\ <
% -\frac{\{3\max(4m-5,0)\pi\}^{2/3}}{4}
% \quad (m=1,2,3,\ldots)
% \end{multline}
% %%%%%%%%%%%%%%%%%%%
%%%%%%%%%%%%%%%%%%%
\begin{equation}
\label{eq165a}
x_{1,m}=-\tfrac{1}{4}\{3(4m-1)\pi\}^{2/3}
\left(1+\tfrac{20}{27}
\{(4m-1)\pi\}^{-2}\right),
\end{equation}
%%%%%%%%%%%%%%%%%%%
and
%%%%%%%%%%%%%%%%%%%
\begin{equation}
\label{eq165b}
x_{2,m}=-\tfrac{1}{4}
\{3\max(4m-5,0)\pi\}^{2/3},
\end{equation}
%%%%%%%%%%%%%%%%%%%
with the $\max$ function being utilized in consideration of the case $m=1$. Again a large number of these zeros could be precomputed and stored.

\appendix
\section{Maple worksheets} 
\label{secA}

\vspace{6pt}

The following Maple worksheet computes the approximation $z_{4}(\nu,m)$ to $z_{m}$ as given by (\ref{eq163}). The notation used here is $zm[j]=z_{m,j}$, ($j=1,2,3,4$), $X=\sigma_{m,0}$, $Y=\zeta_{m,0}$, $Z=z_{m,0}$ and $Am=\mathrm{a}_{m}$.

\vspace{12pt}

\begin{verbatim}
# Input Digits, nu and m

Digits:=20;
nu=10;
m=5;

NU:=rhs(%%):MM:=rhs(%%):

zm[1]:=1/48*X*Z*(10*X^3-6*X*Y-5)/Y^2:

zm[2]:=1/46080*X*Z*(7000*X^9*Z^2-6000*X^7*Y*Z^2+44200*X^9+1080*X^5*Y^2
*Z^2-78560*X^7*Y-4000*X^6*Z^2+37032*X^5*Y^2+1200*X^4*Y*Z^2-2640*X^3
*Y^3-1000*X^4*Y+250*X^3*Z^2+600*X^2*Y^2+250*X*Y-5525)/Y^5:

zm[3]:=8125/41472*Z*(-89451/20800+(Z^4+663/50*Z^2+89451/650)*X^15-396
/325*Y*(Z^4+86491/3960*Z^2+44615/132)*X^13+(-21/26*Z^4-51/10*Z^2)
*X^12+729/1625*Y^2*(Z^4+10471/270*Z^2+7233983/7290)*X^11+36/65*Y*(Z^4
+15187/1200*Z^2-221/240)*X^10+((-378/8125*Z^4-323577/81250*Z^2
-666153/3250)*Y^3+21/130*Z^4)*X^9-243/3250*Y^2*(Z^4+1418/45*Z^2-982
/81)*X^8+((1458/8125*Z^2+9967059/284375)*Y^4+(-9/260*Z^4+51/520*Z^2)
*Y)*X^7+((549/6500*Z^2-13887/32500)*Y^3-1/260*Z^4-51/100*Z^2)*X^6
+(-83673/81250*Y^5+(-81/2600*Z^2+3/520)*Y^2)*X^5+(-51/400*Y+3853
/26000*Y*Z^2+99/3250*Y^4)*X^4+(-9/2600*Y^3+51/800*Z^2)*X^3+491/6500
*X^2*Y^2+51/800*X*Y)*X/Y^8:

zm[4]:=154375/497664*(-4153782141/88524800+(4153782141/691600+Z^6
+102/5*Z^4+1935543/5200*Z^2)*X^21-3927/2470*Y*(Z^6+312353/10472*Z^4
+1446376/1925*Z^2+1037639059/65450)*X^19-20/19*(Z^4+663/50*Z^2+89451
/650)*Z^2*X^18+2187/2470*Y^2*(Z^6+329411/7290*Z^4+242301677/145800
*Z^2+34937103977/729000)*X^17+1386/1235*Y*(Z^6+8296/385*Z^4+1240601
/3696*Z^2-9939/1232)*X^16+((-189/950*Z^6-728253/49400*Z^4-667652763
/771875*Z^2-448855328599/12350000)*Y^3+105/304*Z^6+663/304*Z^4)*X^15
-2187/6175*Y^2*(Z^6+91489/2430*Z^4+3167053/3240*Z^2-223075/2916)
*X^14+891/61750*((Z^6+248551/1650*Z^4+11657643869/693000*Z^2
+67939477043/59400)*Y^3-125/9*(Z^4+52769/4400*Z^2-9061/2640)*Z^2)*Y
*X^13+((189/6175*Z^6+76338/30875*Z^4+32459499/247000*Z^2-7233983
/247000)*Y^3-35/988*Z^6-357/608*Z^4-11271/3040*Z^2)*X^12-187191
/2470000*((Z^4+29145024/80885*Z^2+178654891946/3639825)*Y^3-675
/2311*Z^6-166430/20799*Z^4+526975/41598*Z^2-27625/124794)*Y^2*X^11
-4887/61750*Y*((Z^4+16964203/76020*Z^2-370085/2172)*Y^3-125/1629*Z^6
-30275/6516*Z^4-3356327/52128*Z^2+244205/52128)*X^10+((7413417
/12350000*Z^2+28709041167/86450000)*Y^6+(-1413/197600*Z^4+13077
/39520*Z^2-1473/49400)*Y^3+7/15808*Z^6+357/1520*Z^4)*X^9+229149
/617500*Y^2*((Z^2-27011/4347)*Y^3-76505/611064*Z^4-2832277/611064
*Z^2+542555/305532)*X^8+(-448763193/86450000*Y^7+(-2637/197600*Z^2
+13887/988000)*Y^4+(-15687/316160*Z^4+867/6080*Z^2)*Y)*X^7+(83673
/1235000*Y^6+(125979/1976000*Z^2-3070277/9880000)*Y^3-51/6080*Z^4
-89451/39520*Z^2)*X^6-99/98800*(Y^3+3953/88*Z^2-1105/132)*Y^2*X^5
+(10977/494000*Y^4+(55623/83200*Z^2-89451/158080)*Y)*X^4+(-7931
/1580800*Y^3+1935543/6323200*Z^2)*X^3+533391/1580800*X^2*Y^2+1935543
/6323200*X*Y)*Z*X/Y^11:

AiryAiZeros(MM): evalf(%): Am:=%:

NU^(-2/3)*Am: YY:=evalf(%):

sqrt(Z^2-1)-arcsec(Z)=2/3/NU*(-Am)^(3/2): evalf(%):
ZZ:=fsolve(%,Z=max(2/3*(-YY)^(3/2),1)...2/3*(-YY)^(3/2)+1+Pi/2):

XX:=sqrt(YY/(1-ZZ^2)):

'a'[m]=Am; sigma[m,0]=XX; zeta[m,0]=YY; 'z'[m,0]=ZZ;

Z+sum(zm[s]/nu^(2*s),s=1...4): subs(nu=NU,X=XX,Y=YY,Z=ZZ,%):
'z'[4](nu,m)=%;
\end{verbatim}

\vspace{12pt}

The next Maple worksheet computes the approximation $\tilde{z}_{4}(\nu,m)$ to $\tilde{z}_{m}$ as given by (\ref{eq163}) with $z$ replaced by $\tilde{z}$. In this we have $zt[j]=\tilde{z}_{m,j}$, ($j=1,2,3,4$), $U=\tilde{\sigma}_{m,0}$, $V=\tilde{\zeta}_{m,0}$, $W=\tilde{z}_{m,0}$ and $Amp={\mathrm{a}}^{\prime}_{m}$.

\vspace{12pt}

\begin{verbatim}
# Input Digits, nu and m

Digits:=20;
nu=10;
m=5;

NU:=rhs(%%):MM:=rhs(%%):

zt[1]:=-1/48*U*W*(14*U^3-18*U*V-7)/V^2:

zt[2]:=1/46080*U*W*(13720*U^9*W^2-25200*U^7*V*W^2-58520*U^9+9720*U^5
*V^2*W^2+115360*U^7*V-7840*U^6*W^2-67608*U^5*V^2+5040*U^4*V*W^2+10800
*U^3*V^3-1960*U^4*V+490*U^3*W^2+2520*U^2*V^2+490*U*V+7315)/V^5:

zt[3]:=-22295/41472*(-1910331/1019200+(W^4-627/70*W^2+1910331/31850)
*U^15-1188/455*V*(W^4-157637/16632*W^2+196279/2772)*U^13+(-21/26*W^4
+627/182*W^2)*U^12+6561/3185*V^2*(W^4-9067/810*W^2+9294661/91854)*U^11
+108/91*V*(W^4-3443/720*W^2+209/720)*U^10+((-1458/3185*W^4+37341/4550
*W^2-22849731/222950)*V^3+21/130*W^4)*U^9-2187/6370*V^2*(W^4-1114/135
*W^2+1442/729)*U^8+((-1458/1715*W^2+15418917/780325)*V^4+(-27/364*W^4
+51/520*W^2)*V)*U^7+((-5427/12740*W^2+25353/63700)*V^3-1/260*W^4+627
/1820*W^2)*U^6+(-182979/222950*V^5+(-243/3640*W^2+3/520)*V^2)*U^5
+(-887/3920*V*W^2-81/1274*V^4+627/7280*V)*U^4+(-27/3640*V^3-627/14560
*W^2)*U^3-1423/12740*U^2*V^2-627/14560*U*V)*W*U/V^8:

zt[4]:=593047/497664*(2389852107/173508608+(-2389852107/1355536+W^6
-1254/91*W^4+8296539/50960*W^2)*U^21-1683/494*V*(W^6-453125/31416*W^4
+178487228/962115*W^2-4232928779/1924230)*U^19+(-20/19*W^6+66/7*W^4
-3820662/60515*W^2)*U^18+98415/24206*V^2*(W^6-2438057/153090*W^4
+427411399/1837080*W^2-28968143237/9185400)*U^17+594/247*V*(W^6-2221
/231*W^4+793267/11088*W^2-70753/43120)*U^16+((-3645/1862*W^6+26027325
/677768*W^4-2018197431/2965235*W^2+532087191449/47443760)*V^3+105/304
*W^6-165/112*W^4)*U^15-19683/12103*V^2*(W^6-85453/7290*W^4+4253567
/40824*W^2-196279/26244)*U^14+360855/1186094*((W^6-159683/4950*W^4
+694054643/891000*W^2-9207633077/534600)*V^3-343/243*(W^4-13441/2640
*W^2+779/720)*W^2)*V*U^13+((-73548/12103*W^4+243920277/3388840*W^2
+3645/12103*W^6-9294661/677768)*V^3-35/988*W^6+165/416*W^4-34485/20384
*W^2)*U^12+7512345/9488752*((W^4-5208032/120225*W^2+3605166538
/2318625)*V^3+147/1145*W^6-1816822/1391175*W^4+1604897/1669410*W^2
-71687/5008230)*V^2*U^11+101331/169442*((W^4-3520081/175140*W^2
+2538859/225180)*V^3+245/11259*W^6-139559/135108*W^4+5037109/1080864
*W^2-305767/1080864)*V*U^10+((8041599/6777680*W^2-3002408613/25546640)
*V^6+(59859/387296*W^4-143775/387296*W^2+309/13832)*V^3+7/15808*W^6
-33/208*W^4)*U^9+833247/1694420*V^2*((W^2-190357/72009)*V^3+279415
/740664*W^4-6318887/2221992*W^2+753445/1110996)*U^8+(706289463
/332106320*V^7+(20331/387296*W^2-25353/1936480)*V^4+(227925/3098368
*W^4-561/5824*W^2)*V)*U^7+(182979/3388840*V^6+(1161585/5422144*W^2
-5302207/27110720)*V^3+33/5824*W^4-1910331/1936480*W^2)*U^6+405/193648
*V^2*(V^3+2267/72*W^2-1463/540)*U^5+42543/1355536*V*(V^3+23179997
/1134480*W^2-1485813/189080)*U^4+(22621/3098368*V^3+8296539/61967360
*W^2)*U^3+34539573/108442880*U^2*V^2+8296539/61967360*U*V)*W*U/V^11:

-1/8*3^(2/3)*8^(1/3)*(Pi*(4*MM - 1))^(2/3)
*(1 + 20/(27*Pi^2*(4*MM - 1)^2))
..... -1/8*3^(2/3)*8^(1/3)*(Pi*max(0, 4*MM - 5))^(2/3):
evalf(%): fsolve(AiryAi(1,x)=0,x=%): Amp:=%:

NU^(-2/3)*Amp: VV:=evalf(%):

sqrt(W^2-1)-arcsec(W)=2/3/NU*(-Amp)^(3/2): evalf(%):
WW:=fsolve(%,W=max(2/3*(-VV)^(3/2),1)...2/3*(-VV)^(3/2)+1+Pi/2):

UU:=sqrt(VV/(1-WW^2)):

`#msup(mi("a"),mi("&prime;")))`[m]=Amp;
`#mover(mi("&sigma;"),msup(mi(""),mi("&tilde;")))`[m,0]=UU; 
`#mover(mi("&zeta;"),msup(mi(""),mi("&tilde;")))`[m,0]=VV; 
`#mover(mi(z),msup(mi(""),mi("&tilde;")))`[m,0]=WW; 

W+sum(zt[s]/nu^(2*s),s=1...4): subs(nu=NU,U=UU,V=VV,W=WW,%):
`#mover(mi(z),msup(mi(""),mi("&tilde;")))`[4](nu,m)=%;
\end{verbatim}

\section*{Acknowledgments}
I thank the anonymous referees for many helpful comments.
Financial support from Ministerio de Ciencia e Innovación project PID2021-127252NB-I00 (MCIN/AEI/10.13039/ 501100011033/FEDER, UE) is acknowledged.

\newpage

\bibliographystyle{siamplain}
\bibliography{biblio}
\end{document}